\documentclass{amsart}

\usepackage[latin1]{inputenc}
\usepackage{amssymb,amsmath,amsthm}
\usepackage{amsfonts}
\usepackage{paralist}
\usepackage{booktabs}

\usepackage{url}
\usepackage{hyphenat}

\usepackage{tikz}

\newcommand{\vect}[1]{\ensuremath{\mathbf{#1}}}
\newcommand{\card}[1]{\ensuremath{\lvert{#1}\rvert}}
\newcommand{\multiset}[1]{\ensuremath{\langle{#1}\rangle}}
\newcommand{\dotcup}{\mathop{\mathaccent \cdot \cup}}

\newcommand{\IN}{\ensuremath{\mathbb{N}}}

\newcommand{\nset}[1]{\ensuremath{[{#1}]}}
\newcommand{\couples}[1][n]{\ensuremath{\binom{#1}{2}}} 
\newcommand{\ontuples}[1]{\ensuremath{\underline{#1}}} 
\newcommand{\Anneq}[1][n]{\ensuremath{A^{#1}_{\neq}}}
\newcommand{\Aneq}[1][n]{\ensuremath{A^{#1}_{=}}}
\newcommand{\rev}[1]{\ensuremath{{#1}^\mathrm{rev}}}
\newcommand{\differences}{\ensuremath{\Delta}}

\newcommand{\symm}[1]{\ensuremath{S_{#1}}} 
\newcommand{\cycl}[1]{\ensuremath{Z_{#1}}}         
\newcommand{\dihed}[1]{\ensuremath{D_{#1}}}        
\newcommand{\gensg}[1]{\ensuremath{\langle {#1} \rangle}}
\newcommand{\asc}[1]{\ensuremath{\iota_{#1}}}      
\newcommand{\desc}[1]{\ensuremath{\delta_{#1}}}    
\newcommand{\natcycle}[1]{\ensuremath{\zeta_{#1}}} 

\newcommand{\clUNI}{\ensuremath{\mathsf{UNI}}}
\newcommand{\clOFO}{\ensuremath{\mathsf{OFO}}}
\newcommand{\clCS}{\ensuremath{\mathsf{CS}}}
\newcommand{\clTWOST}{\ensuremath{\mathsf{2ST}}}
\newcommand{\clSYMM}{\ensuremath{\mathsf{SYMM}}}
\newcommand{\clSUPP}{\ensuremath{\mathsf{SUPP}}}

\DeclareMathOperator{\Inv}{Inv}                  
\DeclareMathOperator{\singles}{sng}
\DeclareMathOperator{\indexsingles}{Sng}
\DeclareMathOperator{\cs}{cs}
\DeclareMathOperator{\supp}{supp}                
\DeclareMathOperator{\ofo}{ofo}                  
\DeclareMathOperator{\id}{id}                    
\DeclareMathOperator{\ms}{ms}                    
\DeclareMathOperator{\set}{set}                  
\DeclareMathOperator{\red}{red}                  
\DeclareMathOperator{\deck}{deck}                
\DeclareMathOperator{\Pat}{Pat}
\DeclareMathOperator{\Comp}{Comp}

\theoremstyle{plain}
\newtheorem{theorem}{Theorem}[section]
\newtheorem{proposition}[theorem]{Proposition}
\newtheorem{lemma}[theorem]{Lemma}

\newtheorem{corollary}[theorem]{Corollary}

\numberwithin{claim}{theorem}

\theoremstyle{definition}

\newtheorem{example}[theorem]{Example}
\newtheorem{problem}[theorem]{Problem}
\newtheorem{question}[theorem]{Question}

\theoremstyle{remark}
\newtheorem{remark}[theorem]{Remark}

\begin{document}
\title{Content and singletons bring unique identification minors}
\author{Erkko Lehtonen}
\address{
Technische Universit\"at Dresden \\
Institut f\"ur Algebra \\
01062 Dresden \\
Germany}
\email{Erkko.Lehtonen@tu-dresden.de}
\date{\today}

\begin{abstract}
A new class of functions with a unique identification minor is introduced: functions determined by content and singletons. Relationships between this class with other known classes of functions with a unique identification minor are investigated. Some properties of functions determined by content and singletons are established, especially concerning invariance groups and similarity.
\end{abstract}


\maketitle


\section{Introduction}

This paper reports recent developments in the theory of functions of several arguments, especially in the topic of minors of functions.
A function $f \colon A^n \to B$ is said to be a minor of $g \colon A^m \to B$ if $f$ can be obtained from $g$ by introduction or deletion of inessential arguments, identification of arguments, and permutation of arguments.
The formation of minors is a way of deriving new functions from given ones that has great significance in universal algebra.
Examples of work on minors of functions include the papers of Couceiro and Foldes~\cite{CouFol}, Couceiro, Sch\"olzel, and the current author~\cite{CouLehSch-threshold}, Ekin, Foldes, Hammer, and Hellerstein~\cite{EFHH}, Pippenger~\cite{Pippenger}, Willard~\cite{Willard}, and Zverovich~\cite{Zverovich}.

An important special case of minors are the so\hyp{}called identification minors.
The identification minors of a function $f \colon A^n \to B$ are the $(n-1)$\hyp{}ary functions that are obtained from $f$ by identifying a single pair of arguments.
A function is said to have a unique identification minor if all its identification minors are equal to each other, up to permutation of arguments.
This is an intriguing property of functions that is not quite well understood, and this paper concerns the following open problem.

\begin{problem}
\label{problem}
Characterize the functions with a unique identification minor.
\end{problem}

This problem was previously posed, although using a somewhat different formalism, in the 2010 paper by Bouaziz, Couceiro, and Pouzet~\cite[Problem~2(ii)]{BouCouPou}.
It is well known that the $2$\hyp{}set\hyp{}transitive functions have a unique identification minor (for a proof of this fact, see~\cite[Proposition~4.3]{Lehtonen-totsymm}; this fact is also implicit in the work of Bouaziz, Couceiro, and Pouzet~\cite{BouCouPou}).
It was recently shown by the current author that the functions determined by the order of first occurrence also have this property~\cite[Corollary~5]{Lehtonen-unique}.

In this paper, we describe a previously unknown class of functions that have a unique identification minor: functions determined by content and singletons.
These are functions $f \colon A^n \to B$ whose value at $\vect{a} \in A^n$ depends only on the content of $\vect{a}$ and the order in which the singletons of $\vect{a}$ appear. The content of a tuple $\vect{a}$ is the multiset of elements occurring in $\vect{a}$, and an element $a \in A$ is called a singleton of $\vect{a}$ if it occurs exactly once in $\vect{a}$.

We investigate the relationships between the known classes of functions with a unique identification minor, and we study some invariance properties of functions determined by content and singletons.
However, a definitive answer to Problem~\ref{problem} eludes us; it is not known to us whether there exist further examples of functions with a unique identification minor.

This paper is organised as follows.
We introduce the necessary terminology and notation in Section~\ref{sec:preliminaries}.
In Section~\ref{sec:cs}, we define functions determined by content and singletons and we show that they have a unique identification minor. We verify that this class does indeed provide new examples of functions with a unique identification minor; it is not included in the previously known classes.
In Section~\ref{sec:how}, we explain how functions determined by content and singletons arise as functions with a unique identification minor satisfying a simple additional condition. We do the same for functions determined by the order of first occurrence.
In Section~\ref{sec:permutations}, we study invariance groups of functions determined by content and singletons.
In Section~\ref{sec:similarity} we find distinct functions determined by content and singletons that can be obtained from each other by permutation of arguments.
Finally, in Section~\ref{sec:remarks}, we make some concluding remarks and indicate a few open problems.


\section{Preliminaries}
\label{sec:preliminaries}

\subsection{General notation}

The set of nonnegative integers is denoted by $\IN$, and $\IN_+ := \IN \setminus \{0\}$.
For $n \in \IN_+$, the set $\{1, \dots, n\}$ is denoted by $\nset{n}$.
The power set of a set $A$ is denoted by $\mathcal{P}(A)$.
The set of all $2$\hyp{}element subsets of $\nset{n}$ is denoted by $\couples$.
The symmetric group on $\nset{n}$ is denoted by $\symm{n}$, and its members are sometimes called \emph{$n$\hyp{}permutations.}

As usual, we denote by $A^n$ the set of all $n$\hyp{}tuples over a set $A$, and we write $A^* := \bigcup_{n \geq 0} A^n$.
We will refer to the elements of $A^*$ interchangeably as tuples or strings, and we will often make use of the monoid structure of $A^*$ and concatenate strings.
The empty string is denoted by $\varepsilon$.
We denote tuples (or strings) with bold letters and their components with corresponding italic letters, e.g., $\vect{a} = (a_1, \dots, a_n) = a_1 \dots a_n$.
The length of a string $\vect{a}$ is denoted by $\card{\vect{a}}$.
We say that an element $a \in A$ \emph{occurs} in a string $\vect{a} = (a_1, \dots, a_n) \in A^n$ if $a_i = a$ for some $i \in \nset{n}$.

We denote by $\Anneq$ the set of all $n$\hyp{}tuples on $A$ with no repeated entries, i.e., $\Anneq = \{(a_1, \dots, a_n) \in A^n : a_i = a_j \implies i = j\}$. Note that $\Anneq = \emptyset$ whenever $n > \card{A}$.
Let $A^\sharp := \bigcup_{n \geq 0} \Anneq$, i.e., $A^\sharp$ is the set of all strings over $A$ with no repeated symbols.
Furthermore, write $\Aneq := A^n \setminus \Anneq$. 

Let $\vect{a} = (a_1, \dots, a_n) \in A^n$ and let $\tau \colon \nset{m} \to \nset{n}$.
Since the $n$\hyp{}tuple $\vect{a}$ is formally a map $\nset{n} \to A$, we may compose it with $\sigma$, and we obtain $\vect{a} \circ \tau \colon \nset{m} \to A$, which is the $m$\hyp{}tuple $(a_{\tau(1)}, \dots, a_{\tau(m)})$.
In order to simplify notation, we will abbreviate $\vect{a} \circ \tau$ as $\vect{a} \tau$.
The map $\tau$ induces a map $\ontuples{\tau} \colon A^n \to A^m$ by the rule $\ontuples{\tau}(\vect{a}) = \vect{a} \tau$ for all $\vect{a} \in A^n$.

\subsection{Functions and minors}

Let $A$ and $B$ be arbitrary nonempty sets. A \emph{function \textup{(}of several arguments\textup{)} from $A$ to $B$} is a mapping $f \colon A^n \to B$ for some positive integer $n$ called the \emph{arity} of $f$.
The set of all functions of several arguments from $A$ to $B$ is denoted by $\mathcal{F}_{AB}$.
For any set $\mathcal{C} \subseteq \mathcal{F}_{AB}$, the \emph{$n$\hyp{}ary part} of $\mathcal{C}$ is the set of the $n$\hyp{}ary members of $\mathcal{C}$ and it is denoted by $\mathcal{C}^{(n)}$.
We also write $\mathcal{C}^{(\geq k)} := \bigcup_{n \geq k} \mathcal{C}^{(n)}$.

Let $f \colon A^n \to B$ and $g \colon A^m \to B$.
We say that $f$ is a \emph{minor} of $g$, and we write $f \leq g$, if there exists a map $\tau \colon \nset{m} \to \nset{n}$ such that $f = g \circ \ontuples{\tau}$, i.e., $f(\vect{a}) = g(\vect{a} \tau)$ for all $\vect{a} \in A^n$.
We say that $f$ and $g$ are \emph{equivalent,} and we write $f \equiv g$, if $f \leq g$ and $g \leq f$.
The minor relation $\leq$ is a quasiorder and $\equiv$ is an equivalence relation on $\mathcal{F}_{AB}$.
The induced partial order on $\mathcal{F}_{AB} / {\equiv}$ will also be denoted by $\leq$.
We will often consider the following refinement of equivalence:
we say that $f$ and $g$ are \emph{similar,} and we write $f \simeq g$, if $f$ and $g$ have the same arity ($n = m$) and $f \equiv g$.
It is easy to verify that $f \simeq g$ if and only if there exists a bijection $\tau \colon \nset{n} \to \nset{n}$ such that $f = g \circ \ontuples{\tau}$.

Of particular importance are the minors obtained by identifying a single pair of arguments.
Assume that $n \geq 2$, and let $f \colon A^n \to B$.
For each $I \in \couples$, define the function $f_I \colon A^{n-1} \to B$ as $f_I = f \circ \ontuples{\delta_I}$, where $\delta_I \colon \nset{n} \to \nset{n-1}$ is given by the rule
\[
\delta_I(i) =
\begin{cases}
i,      & \text{if $i < \max I$,} \\
\min I, & \text{if $i = \max I$,} \\
i - 1,  & \text{if $i > \max I$.}
\end{cases}
\]
More explicitly, if $I = \{i, j\}$ with $i < j$, then
\[
f_I(a_1, \dots, a_{n-1}) = f(a_1, \dots a_{j-1}, a_i, a_j, \dots, a_{n-1}).
\]
Note that $a_i$ appears twice on the right side of the above equality: both at the $i$\hyp{}th and at the $j$\hyp{}th positions.
The functions $f_I$ ($I \in \couples$) are referred to as the \emph{identification minors} of $f$.

\subsection{Invariance groups}
\label{subsec:invariance}

A function $f \colon A^n \to B$ is \emph{invariant} under a permutation $\sigma \in \symm{n}$ if $f = f \circ \ontuples{\sigma}$.
The set of all permutations under which $f$ is invariant constitutes a permutation group, and it is called the \emph{invariance group} of $f$ and denoted by $\Inv f$.
A function is \emph{totally symmetric,} if its invariance group is the full symmetric group $\symm{n}$.
A function is \emph{$2$\hyp{}set\hyp{}transitive,} if its invariance group is $2$\hyp{}set\hyp{}transitive. Recall that a permutation group $G \leq \symm{n}$ is \emph{$2$\hyp{}set\hyp{}transitive,} if it acts transitively on the $2$\hyp{}element subsets of $\nset{n}$, i.e., for all $I, J \in \couples$, there exists $\sigma \in G$ such that $\sigma(I) = J$.
The invariance groups of functions have caught much attention of researchers; see, for example, the papers by Grech and Kisielewicz~\cite{GreKis}, Horv\'ath, Makay, P\"oschel, and Waldhauser~\cite{HMPW}, and Kisielewicz~\cite{Kisielewicz}.

The notion of invariance extends immediately to partial functions.
We say that a partial function $f \colon S \to B$, $S \subseteq A^n$, is \emph{invariant} under $\sigma \in \symm{n}$ if for every $\vect{a} \in S$, it holds that $\vect{a} \sigma \in S$ and $f(\vect{a}) = f(\vect{a} \sigma)$.
In this paper, we will often consider partial functions whose domain is $\Aneq$ (or, of course, $A^n$), but other domains are also possible.

\subsection{Functions determined by $\phi$}
\label{subsec:detby}

We will classify functions by their decomposability through certain mappings.
Let $\phi \colon A^* \to X$ for some set $X$.
We say that a function $f \colon A^n \to B$ is \emph{determined by $\phi$} if there exists a map $f^* \colon X \to B$ such that $f = f^* \circ \phi|_{A^n}$.
This definition extends immediately to partial functions: a partial function $f \colon S \to B$, $S \subseteq A^n$, is \emph{determined by $\phi$} if $f = f^* \circ \phi|_S$ for some $f^* \colon X \to B$.

In this paper, the mapping $\phi$ used in the expression ``determined by $\phi$'' will be one of the following four maps: $\supp$, $\ofo$, $\ms$, $\cs$.
The first two of these mappings are defined below, and the last two will be defined in Section~\ref{sec:cs}.

Following the definition given by Berman and Kisielewicz~\cite{BerKis}, $\supp \colon A^* \to \mathcal{P}(A)$ maps every tuple to the set of its entries, i.e., $\supp(a_1, \dots, a_n) = \{a_1, \dots, a_n\}$.

As defined in~\cite{Lehtonen-unique}, the map $\ofo \colon A^* \to A^\sharp$ is given by the following rule: $\ofo$ maps each tuple $\vect{a}$ to the tuple obtained from $\vect{a}$ by deleting all repeated occurrences of symbols, retaining only the first occurrence of each symbol; in other words, $\ofo(\vect{a})$ lists the symbols occurring in $\vect{a}$ in the order of first occurrence (hence the initialism ofo).
We also say that a function determined by $\ofo$ is \emph{determined by the order of first occurrence.}
The mapping $\ofo$ has the following simple but very fundamental property: $\ofo(\vect{a}) = \ofo(\vect{a} \delta_I)$ for every $\vect{a} \in A^{n-1}$ and for every $I \in \couples$.

\subsection{Multisets}

A \emph{multiset} over a set $A$ is an ordered pair $M = (A, \chi_M)$, where $\chi_M \colon A \to \IN$ is a map called a \emph{multiplicity function.}
The value $\chi_M(a)$ is referred to as the \emph{multiplicity} of $a$ in $M$.
The \emph{support set} of a multiset $M$, denoted by $\set(M)$, is the set of those elements that have nonzero multiplicity in $M$, i.e., $\set(M) := \{a \in A : \chi_M(a) > 0\}$.
A multiset is \emph{finite} if its support set is finite.
If $M$ is finite, then the sum $\sum_{a \in A} \chi_M(a)$ is a well\hyp{}defined integer that is denoted by $\card{M}$ and called the \emph{cardinality} of $M$.
The set of all finite multisets over $A$ is denoted by $\mathcal{M}(A)$.
In this paper, we will consider only finite multisets and we will refer to them simply as multisets.

We may describe a finite multiset as a list enclosed in angle brackets in which the number of occurrences of each element equals its multiplicity (the order does not matter), e.g., $\multiset{1,1,1,2,3,3}$ is the multiset $(A, \chi_M)$ satisfying $\chi_M(1) = 3$, $\chi_M(2) = 1$, $\chi_M(3) = 2$, and $\chi_M(a) = 0$ for all $a \in A \setminus \{1, 2, 3\}$.
We will use the shorthand $a^m$ for $m$ occurrences of $a$. Thus, we can describe the above multiset equivalently as $\multiset{1^3, 2, 3^2}$.

The \emph{join} of multisets $M_1$ and $M_2$ over $A$ is the multiset $M_1 \dotcup M_2$ given by the multiplicity function $\chi_{M_1 \dotcup M_2}(a) = \chi_{M_1}(a) + \chi_{M_2}(a)$ for all $a \in A$.


\section{Functions determined by content and singletons}
\label{sec:cs}

A function $f \colon A^n \to B$ is said to have a \emph{unique identification minor,} if $f_I \simeq f_J$ for all $I, J \in \couples$.
This condition is equivalent to the following: there exists a function $h \colon A^{n-1} \to B$ and a family $(\rho_I)_{I \in \couples}$ of permutations $\rho_I \in \symm{n-1}$ such that $f_I = h \circ \ontuples{\rho_I}$ for all $I \in \couples$.
Note that if $f \simeq g$ and $f$ has a unique identification minor, then also $g$ has a unique identification minor.

As mentioned in the introduction, it is known that the $2$\hyp{}set\hyp{}transitive functions and the functions determined by the order of first occurrence have a unique identification minor (see~\cite{BouCouPou}, \cite[Proposition~4.3]{Lehtonen-totsymm}, \cite[Corollary~5]{Lehtonen-unique}).
In this section, we will describe another class of functions with a unique identification minor that was not, to the best of the author's knowledge, previously known.

Let $\vect{a} = (a_1, \dots, a_n) \in A^n$.
We refer to the multiset $\multiset{a_1, \dots, a_n}$ of entries of $\vect{a}$ as the \emph{content} of $\vect{a}$.
Let $\ms \colon A^* \to \mathcal{M}(A)$ be the map $(a_1, \dots, a_n) \mapsto \multiset{a_1, \dots, a_n}$.
An element $a \in A$ is called a \emph{singleton} of $\vect{a}$, if $a$ occurs exactly once in $\vect{a}$.
Define the mapping $\singles \colon A^* \to A^\sharp$ by the following rule:
$\singles(\vect{a})$ is the tuple that lists the singletons of $\vect{a}$ in the order of their occurrence in $\vect{a}$.

\begin{remark}
\label{rem:ms-totsymm}
A function is determined by $\ms$ if and only if it is totally symmetric.
\end{remark}

Let
\[
\mathcal{Z}(A) := \{(M, \vect{a}) \in \mathcal{M}(A) \times A^\sharp : \forall a \in A \, (\text{$a$ occurs in $\vect{a}$} \iff \chi_M(a) = 1)\},
\]
and for each $n, \ell \in \IN$, let
\begin{align*}
\mathcal{Z}^{(n)}(A) &:= \{(M, \vect{a}) \in \mathcal{Z}(A) : \card{M} = n\}, \\
\mathcal{Z}^{(n)}_\ell(A) &:= \{(M, \vect{a}) \in \mathcal{Z}(A) : \text{$\card{M} = n$ and $\card{\vect{a}} = \ell$}\}.
\end{align*}
Note that if $\card{A} = k$, then $\mathcal{Z}^{(n)}_\ell(A) \neq \emptyset$ if and only if $\ell \in Z(k,n)$, where
\[
Z(k,n) :=
\begin{cases}
\{0, 1, \dots, k-1\}, & \text{if $k < n$,} \\
\{0, 1, \dots, n\} \setminus \{n-1\}, & \text{if $k \geq n$.}
\end{cases}
\]

Define the mapping $\cs \colon A^* \to \mathcal{Z}(A)$ by the rule
$\cs(\vect{a}) := (\ms(\vect{a}), \singles(\vect{a}))$.
(The name of the function $\cs$ is an initialism of ``content and singletons''.)
Note that $\cs$ is surjective onto $\mathcal{Z}(A)$, and the range of the restriction ${\cs}|_{A^n}$ is $\mathcal{Z}^{(n)}(A)$.
We also say that a function determined by $\cs$ is \emph{determined by content and singletons.}

\begin{example}
In order to illustrate the mappings $\ms$, $\singles$ and $\cs$, let $A$ be the set of lower\hyp{}case letters of the English alphabet. Below are shown the images under $\cs$ of some strings over $A$.
\begin{align*}
&\cs(\mathsf{mathematician}) = (\multiset{\mathsf{a}^3, \mathsf{c}, \mathsf{e}, \mathsf{h}, \mathsf{i}^2, \mathsf{m}^2, \mathsf{n}, \mathsf{t}^2}, \mathsf{hecn}) \\
&\cs(\mathsf{circumlocution}) = (\multiset{\mathsf{c}^3, \mathsf{i}^2, \mathsf{l}, \mathsf{m}, \mathsf{o}^2, \mathsf{n}, \mathsf{r}, \mathsf{t}, \mathsf{u}^2}, \mathsf{rmltn}) \\
&\cs(\mathsf{ambidextrously}) = (\multiset{\mathsf{a}, \mathsf{b}, \mathsf{d}, \mathsf{e}, \mathsf{i}, \mathsf{l}, \mathsf{m}, \mathsf{o}, \mathsf{r}, \mathsf{s}, \mathsf{t}, \mathsf{u}, \mathsf{x}, \mathsf{y}}, \mathsf{ambidextrously}) \\
&\cs(\mathsf{unprosperousness}) = (\multiset{\mathsf{e}^2, \mathsf{n}^2, \mathsf{o}^2, \mathsf{p}^2, \mathsf{r}^2, \mathsf{s}^4, \mathsf{u}^2}, \varepsilon) \\
\end{align*}
\end{example}

\begin{lemma}
\label{lem:minors-cs}
Let $f \colon A^n \to B$, $f = f^* \colon {\cs}|_{A^n}$ for some $f^* \colon \mathcal{Z}(A) \to B$.
Then for every $I = \{i, j\} \in \couples$ with $i < j$, it holds that $f_I = f^* \circ {\cs^{n-1}_i}$, where the map
$\cs^{n-1}_i \colon A^{n-1} \to \mathcal{Z}(A)$ is given by the rule
\[
\cs^{n-1}_i(a_1, \dots, a_{n-1}) := (\ms(a_1, \dots, a_{n-1}) \dotcup \multiset{a_i}, \singles(a_1, \dots, a_{i-1}, a_{i+1}, \dots, a_{n-1})),
\]
for all $(a_1, \dots, a_{n-1}) \in A^{n-1}$.
\end{lemma}

\begin{proof}
Straighforward calculations show that for all $\vect{a} = (a_1, \dots, a_{n-1}) \in A^{n-1}$,
\begin{align*}
f_I(\vect{a})
&= f(\vect{a} \delta_I)
= f^* \circ \cs(\vect{a} \delta_I)
= f^* \circ (\ms(\vect{a} \delta_I), \singles(\vect{a} \delta_I)) \\
&= f^* \circ (\ms(\vect{a}) \dotcup \multiset{a_i}, \singles(a_1, \dots, a_{i-1}, a_{i+1}, \dots, a_{n-1}))
= f^* \circ \cs^{n-1}_i(\vect{a}).
\qedhere
\end{align*}
\end{proof}

\begin{lemma}
\label{lem:cs-unique}
Any function determined by $\cs$ has a unique identification minor.
\end{lemma}

\begin{proof}
Assume that $f = f^* \circ {\cs}|_{A^n}$ for some $f^* \colon \mathcal{Z}(A) \to B$.
Let $I = \{i, j\} \in \couples$ with $i < j$, and let $K := \{n - 1, n\}$.
Let $\xi_i \in \symm{n-1}$ be the cycle $(i \; i + 1\; \dots \; n - 1)$.
For every $\vect{a} = (a_1, \dots, a_{i-1}) \in A^{n-1}$, it clearly holds that
\[
\cs^{n-1}_i(a_1, \dots, a_{n-1})
= \cs^{n-1}_{n-1}(a_1, \dots, a_{i-1}, a_{i+1}, \dots, a_{n-1}, a_i),
\]
i.e., $\cs^{n-1}_i(\vect{a}) = \cs^{n-1}_{n-1}(\vect{a} \xi_i)$.
Consequently, by Lemma~\ref{lem:minors-cs}, for all $\vect{a} \in A^{n-1}$,
\[
f_I(\vect{a})
= f^* \circ {\cs^{n-1}_i}(\vect{a})
= f^* \circ {\cs^{n-1}_{n-1}}(\vect{a} \xi_i)
= f_K(\vect{a} \xi_i).
\]
This means that $f_I \simeq f_K$.
Since the choice of $I \in \couples$ was arbitrary, the symmetry and transitivity of $\simeq$ imply that $f_I \simeq f_J$ for all $I, J \in \couples$.
\end{proof}

Before proceeding any further, let us verify that we have actually discovered something new. Does the class of functions determined by $\cs$ provide examples of functions that are not already subsumed by the previously known classes of functions with a unique identification minor (namely, $2$\hyp{}set\hyp{}transitive functions and functions determined by $\ofo$)?
This question is answered in the following two lemmas.

\begin{lemma}
\label{lem:cs-new-2}
Assume that $\card{A} = 2$ and $f \colon A^n \to B$ is determined by $\cs$.
Then $n \leq 2$ or $f$ is totally symmetric.
\end{lemma}

\begin{proof}
Assume that $A = \{0, 1\}$, $n \geq 3$, and $f = f^* \circ {\cs}|_{A^n}$ for some $f^* \colon \mathcal{Z}^{(n)}(A) \to B$.
It is easy to see that the range $\mathcal{Z}^{(n)}(A)$ of ${\cs}|_{A^n}$ equals
\[
\{(\multiset{0^n}, \varepsilon),
(\multiset{1^n}, \varepsilon),
(\multiset{0^{n-1}, 1}, 1),
(\multiset{0, 1^{n-1}}, 0)\}
\cup
\{(\multiset{0^\ell, 1^{n-\ell}}, \varepsilon) : 2 \leq \ell \leq n - 2\}.
\]
It is clear that if $M \in \mathcal{M}(A)$ and $\vect{u}, \vect{v} \in A^\sharp$ are such that $(M, \vect{u}), (M, \vect{v}) \in \mathcal{Z}^{(n)}(A)$, then $\vect{u} = \vect{v}$.
Therefore each couple $(M, \vect{u}) \in \mathcal{Z}^{(n)}(A)$ is uniquely specified by its first component $M$ alone, while the second component $\vect{u}$ is irrelevant.
Therefore, $f(\vect{a})$ depends only on $\ms(\vect{a})$, i.e., $f$ is determined by $\ms$. By Remark~\ref{rem:ms-totsymm}, $f$ is totally symmetric.
\end{proof}

The totally symmetric functions are $2$\hyp{}set\hyp{}transitive, and it is obvious that all unary and binary functions have a unique identification minor. Thus, functions determined by $\cs$ did not really bring anything new when $\card{A} = 2$.
The situation is completely different when $\card{A} \geq 3$.

\begin{lemma}
\label{lem:cs-new}
Assume that $A$ and $B$ are sets with $\card{A} = k \geq 3$ and $\card{B} \geq 2$.
Then, for every $n \geq k + 1$, there exists a function $f \colon A^n \to B$ that is determined by $\cs$ but
it is not similar to any function determined by $\ofo$, nor is it $2$\hyp{}set\hyp{}transitive.
\end{lemma}

\begin{proof}
Assume, without loss of generality, that $A = \{1, \dots, k\}$ and $0, 1 \in B$.
Let $f^* \colon \mathcal{Z}^{(n)}(A) \to B$ be the function that maps
$(\multiset{1^{n-k+1}, 2, 3, \dots, k}, 23 \dots k)$ to $1$ and everything else to $0$.
Define $f \colon A^n \to B$ as $f = f^* \circ {\cs}|_{A^n}$.
Then $f$ is determined by $\cs$ by definition.
Note that $f^{-1}(1)$ comprises precisely those $n$\hyp{}tuples that have $n-k+1$ occurrences of $1$ and exactly one occurrence of every other element of $A$ and in which the entries distinct from $1$ appear in increasing order.

Suppose, to the contrary,
that $f$ is similar to a function determined by $\ofo$,
that is, $f = f' \circ {\ofo}|_{A^n} \circ \ontuples{\sigma}$ for some $f' \colon A^\sharp \to B$ and $\sigma \in \symm{n}$. Let $\vect{u} = 1 \dots 1 2 3 \dots k$, where the element $1$ is repeated $n - k + 1$ times. It is clear that there exists a tuple $\vect{v} \in A^n$ such that $\vect{v}$ has at least two occurrences of $2$ and $\ofo(\vect{u} \sigma) = \ofo(\vect{v} \sigma)$ (take, for example, the tuple $(\vect{w} 2 \dots 2) \sigma^{-1}$, where $\vect{w} = \ofo(\vect{u} \sigma)$). Since $f = f' \circ {\ofo}|_{A^n} \circ \ontuples{\sigma}$, the equality $f(\vect{u}) = f'(\ofo(\vect{u} \sigma)) = f'(\ofo(\vect{v} \sigma)) = f(\vect{v})$ holds.
On the other hand, since $\cs(\vect{u}) = (\multiset{1^{n-k+1}, 2, 3, \dots, k}, 23 \dots k) \neq \cs(\vect{v})$, we have $f(\vect{u}) = f^*(\cs(\vect{u})) = 1 \neq 0 = f^*(\cs(\vect{v})) = f(\vect{v})$.
We have reached a contradiction.

We claim that the invariance group of $f$ is trivial, and hence $f$ is not $2$\hyp{}set\hyp{}transitive.
Let $\sigma \in \Inv f$.
Suppose, to the contrary, that $\sigma \neq \id$.
Then there exist $i, j \in \nset{n}$ such that $i < j$ and $\sigma(j) < \sigma(i)$.
It is easy to see that $f^{-1}(1)$ contains a tuple $\vect{u}$ such that $u_{\sigma(j)} = p$ and $u_{\sigma(i)} = q$ for some $p$ and $q$ distinct from $1$.
In fact, we have $p < q$.
But then the relative order of the entries $p$ and $q$ is reversed in $\vect{u} \sigma$, so $f(\vect{u} \sigma) = 0 \neq 1 = f(\vect{u})$, a contradiction.
\end{proof}

\begin{proposition}[{\cite[Proposition~7]{Lehtonen-unique}}]
\label{prop:2st-ofo}
Assume that $n > \card{A} + 1$, and let $f \colon A^n \to B$. Then the following conditions are equivalent.
\begin{enumerate}[\rm (i)]
\item $f$ is totally symmetric and determined by $\ofo$.
\item $f$ is $2$\hyp{}set\hyp{}transitive and determined by $\ofo$.
\item $f$ is determined by $\ofo$ and for all $I, J \in \couples$, there exists a bijection \linebreak $\pi_{IJ} \colon \nset{n-1} \to \nset{n-1}$ such that $\pi_{IJ}(\min J) = \min I$ and $f(\vect{a} \delta_I) = f(\vect{a} \pi_{IJ} \delta_J)$ for all $\vect{a} \in A^{n-1}$.
\item $f$ is determined by $\supp$.
\end{enumerate}
\end{proposition}

\begin{proposition}
\label{prop:ofo-and-cs}
Let $f \colon A^n \to B$.
Then $f$ is similar to a function determined by $\ofo$ and similar to a function determined by $\cs$
if and only if $f|_{\Aneq}$ is determined by $\supp$.
\end{proposition}

\begin{proof}
Assume first that $f|_{\Aneq}$ is determined by $\supp$, i.e., there exists $f^\dagger \colon \mathcal{P}(A) \to B$ such that $f|_{\Aneq} = f^\dagger \circ {\supp}|_{\Aneq}$.
Let $f^* \colon A^\sharp \to B$ be any function satisfying
\[
f^*(\vect{b}) =
\begin{cases}
f^\dagger(\supp(\vect{b})), & \text{if $\card{\vect{b}} < n$,} \\
f(\vect{b}), & \text{if $\card{\vect{b}} = n$,}
\end{cases}
\qquad
\vect{b} \in A^\sharp.
\]
(Note that for every $\vect{b} \in A^\sharp$, it holds that $\card{\vect{b}} = \card{\supp(\vect{b})}$. Note also that the second branch of the defining condition of $f^*$ is irrelevant if $n > \card{A}$.)
Let $\vect{a} \in A^n$.
If $\vect{a} \in \Aneq$, then $\card{\supp(\vect{a})} < n$ and, since $\supp(\ofo(\vect{a})) = \supp(\vect{a})$, we have
\[
f^*(\ofo(\vect{a}))
= f^\dagger(\supp(\ofo(\vect{a})))
= f^\dagger(\supp(\vect{a}))
= f(\vect{a}).
\]
If $\vect{a} \in \Anneq$, then $\ofo(\vect{a}) = \vect{a}$ and we have $f^*(\ofo(\vect{a})) = f(\vect{a})$ also in this case.
Thus, $f$ is determined by $\ofo$.

In order to show that $f$ is also determined by $\cs$,
let $f' \colon \mathcal{Z}^{(n)}(A) \to B$ be any function satisfying
\[
f'(M, \vect{b}) =
\begin{cases}
f^\dagger(\set(M)), & \text{if $\card{\vect{b}} < n$,} \\
f(\vect{b}), & \text{if $\card{\vect{b}} = n$.}
\end{cases}
\]
Let $\vect{a} \in A^n$.
If $\vect{a} \in \Aneq$, then $\card{\singles(\vect{a})} < n$ and we have
\[
f'(\cs(\vect{a}))
= f'(\ms(\vect{a}), \singles(\vect{a}))
= f^\dagger(\set(\ms(\vect{a})))
= f^\dagger(\supp(\vect{a}))
= f(\vect{a}).
\]
If $\vect{a} \in \Anneq$, then $\singles(\vect{a}) = \vect{a}$, so $\card{\singles(\vect{a})} = n$ and we have
\[
f'(\cs(\vect{a}))
= f'(\ms(\vect{a}), \vect{a})
= f(\vect{a}).
\]
We have shown that $f$ is determined by both $\ofo$ and $\cs$.
Since $f \simeq f$, we arrive to the desired conclusion.

Assume then that $f$ is similar to a function determined by $\ofo$ and similar to a function determined by $\cs$.
We may assume, without loss of generality, that $f = f^* \circ {\ofo}|_{A^n} = f' \circ {\cs}|_{A^n} \circ \ontuples{\pi}$ for some $f^* \colon A^\sharp \to B$, $f' \colon \mathcal{Z}^{(n)}(A) \to B$, $\pi \in \symm{n}$.

Since $\Aneq[1] = \emptyset$ and $\Aneq[2] = \{(a,a) \mid a \in A\}$, it is obvious that $f|_{\Aneq}$ is determined by $\supp$ when $1 \leq n \leq 2$.
Consider then the case when $n = 3$.
Since $f = f^* \circ {\ofo}|_{A^n}$, it holds that for all $a, b \in A$ with $a \neq b$,
\begin{gather*}
f(a,a,b) = f(a,b,a) = f(a,b,b) = f^*(ab),
\\
f(b,a,a) = f(b,a,b) = f(b,b,a) = f^*(ba).
\end{gather*}
Since $f = f' \circ {\cs}|_{A^n} \circ \ontuples{\pi}$, it holds that for all $a, b \in A$ with $a \neq b$,
\begin{gather*}
f(a,a,b) = f(a,b,a) = f(b,a,a) = f'(\multiset{a^2,b},b),
\\
f(b,b,a) = f(b,a,b) = f(a,b,b) = f'(\multiset{a,b^2},a).
\end{gather*}
Consequently,
\[
f(a,a,b) = f(a,b,a) = f(a,b,b) = f(b,a,a) = f(b,a,b) = f(b,b,a),
\]
and we conclude that $f|_{\Aneq} = f^\dagger \circ {\supp}|_{\Aneq}$, where $f^\dagger \colon \mathcal{P}(A) \to B$ is any map satisfying $f^\dagger(\{a\}) = f(a,a,a)$ and $f^\dagger(\{a,b\}) = f(a,a,b)$ for all $a, b \in A$ with $a \neq b$.

Finally, assume that $n \geq 4$.
Let
\begin{align*}
U &:= \{\vect{a} \in A^n : \card{\supp(\vect{a})} \leq n - 2\}, \\
V &:= \{\vect{a} \in A^n : \card{\supp(\vect{a})} = n - 1\}.
\end{align*}
We are going to show that both $f|_U$ and $f|_V$ are determined by $\supp$, i.e., there exist $f^\dagger, f^\ddagger \colon \mathcal{P}(A) \to B$ such that $f|_U = f^\dagger \circ {\supp}|_U$ and $f|_V = f^\ddagger \circ {\supp}|_V$.
From this it will follow that $f|_{\Aneq}$ is determined by $\supp$, because
$f|_{\Aneq} = f^\flat \circ {\supp}|_{\Aneq}$, where $f^\flat \colon \mathcal{P}(A) \to B$ is defined as
\[
f^\flat(S) =
\begin{cases}
f^\dagger(S), & \text{if $\card{S} \leq n - 2$,} \\
f^\ddagger(S), & \text{otherwise.}
\end{cases}
\]

We show first that $f|_U$ is determined by $\supp$.
Let $\alpha, \beta \in A$, $\alpha \neq \beta$, and let $\vect{u} \in A^{n-4}$ be a string with no occurrence of $\beta$.
It is easy to see that $\cs(\ontuples{\sigma}(\alpha \alpha \vect{u} \beta \beta)) = \cs(\ontuples{\sigma}(\beta \alpha \vect{u} \alpha \beta)) = \cs(\ontuples{\sigma}(\alpha \beta \vect{u} \alpha \beta))$, for any permutation $\sigma \in \symm{n}$.
In particular, since $f = f' \circ {\cs}|_{A^n} \circ \ontuples{\pi}$, this implies that $f(\alpha \alpha \vect{u} \beta \beta) = f(\beta \alpha \vect{u} \alpha \beta) = f(\alpha \beta \vect{u} \alpha \beta)$.
Note also that there exists a string $\vect{v} \in A^\sharp$ with no occurrence of $\alpha$ nor $\beta$ such that
$\ofo(\alpha \alpha \vect{u} \beta \beta) = \alpha \vect{v} \beta$,
$\ofo(\beta \alpha \vect{u} \alpha \beta) = \beta \alpha \vect{v}$ and
$\ofo(\alpha \beta \vect{u} \alpha \beta) = \alpha \beta \vect{v}$.
Since $f = f^* \circ {\ofo}|_{A^n}$, we must have $f^*(\alpha \vect{v} \beta) = f^*(\beta \alpha \vect{v}) = f^*(\alpha \beta \vect{v})$.
Now $\alpha \vect{v} \beta, \beta \alpha \vect{v}, \alpha \beta \vect{v} \in \Anneq[m]$ for some $m \leq n - 2$.
Since the choice of $\alpha$, $\beta$ and $\vect{u}$ was arbitrary, we conclude that $f^*|_{\Anneq[m]}$ is invariant under the permutations $\sigma = (1 \; 2 \; \dots \; m)$ and $\tau = (1 \; 2)$.
Since $\{\sigma, \tau\}$ generates the full symmetric group $\symm{m}$, this implies that $f^*|_{\Anneq[m]}$ is totally symmetric whenever $m \leq n - 2$, i.e., $f^*(\vect{b}) = f^*(\vect{b} \rho)$ for any $\vect{b} \in \Anneq[m]$ and any permutation $\rho \in \symm{m}$.
In other words, $f^*|_{\Anneq[m]}(\vect{b})$ depends only on $\supp(\vect{b})$.
Consequently, there exists a mapping $f^\dagger \colon \mathcal{P}(A) \to B$ that satisfies
$f^\dagger(\supp(\vect{b})) = f^*(\vect{b})$ for every $\vect{b} \in A^\sharp$ with $\card{\vect{b}} \leq n - 2$.
Since $\supp(\vect{a}) = \supp(\ofo(\vect{a}))$, we have
\[
f^\dagger(\supp(\vect{a}))
= f^\dagger(\supp(\ofo(\vect{a})))
= f^*(\ofo(\vect{a}))
= f(\vect{a})
\]
for all $\vect{a} \in A^n$ with $\card{\supp(\vect{a})} \leq n - 2$,
i.e., $f|_U$ is determined by $\supp$.

It remains to show that $f|_V$ is determined by $\supp$.
Let $p, q, r, x, y, z \in \nset{n}$ be elements satisfying $p < q < r$, $x < y < z$ and $\{\pi(p), \pi(q), \pi(r)\} = \{x, y, z\}$.
Let $\vect{a}, \vect{b}, \vect{c} \in A^n$ be tuples such that $L := \supp(\vect{a}) = \supp(\vect{b}) = \supp(\vect{c})$, $\card{L} = n - 1$, $a_i = b_i$ whenever $i \neq z$, $b_j = c_j$ whenever $j \neq y$ and
$\alpha := a_x = a_z = b_x = c_x = c_y$ and
$\beta := a_y = b_y = b_z = c_z$.
It clearly holds that $\ofo(\vect{a}) = \ofo(\vect{b}) = \ofo(\vect{c})$.
Since $f = f^* \circ {\ofo}|_{A^n}$, we have $f(\vect{a}) = f(\vect{b}) = f(\vect{c})$.
Consequently, since $f = f' \circ {\cs}|_{A^n} \circ \ontuples{\pi}$, the mapping $f'$ must take the same value at points $\cs(\ontuples{\pi}(\vect{a}))$, $\cs(\ontuples{\pi}(\vect{b}))$ and $\cs(\ontuples{\pi}(\vect{c}))$.

Since $\vect{a}$ and $\vect{b}$ differ only at the $z$\hyp{}th component, $\ontuples{\pi}(\vect{a})$ and $\ontuples{\pi}(\vect{b})$ differ only at the $\pi^{-1}(z)$\hyp{}th component, and $\pi^{-1}(z) \in \{p, q, r\}$.
Similarly, since $\vect{b}$ and $\vect{c}$ differ only at the $y$\hyp{}th component, $\ontuples{\pi}(\vect{b})$ and $\ontuples{\pi}(\vect{c})$ differ only at the $\pi^{-1}(y)$\hyp{}th component, and $\pi^{-1}(y) \in \{p, q, r\}$.
Therefore, there exist strings $\vect{u}_1, \vect{u}_2, \vect{u}_3, \vect{u}_4 \in A^*$ such that $\card{\vect{u}_1} = p - 1$, $\card{\vect{u}_2} = q - p - 1$, $\card{\vect{u}_3} = r - q - 1$, $\card{\vect{u}_4} = n - r$
and
\begin{align*}
\ontuples{\pi}(\vect{a}) &= \vect{u}_1 a_{\pi(p)} \vect{u}_2 a_{\pi(q)} \vect{u}_3 a_{\pi(r)} \vect{u}_4, \\
\ontuples{\pi}(\vect{b}) &= \vect{u}_1 b_{\pi(p)} \vect{u}_2 b_{\pi(q)} \vect{u}_3 b_{\pi(r)} \vect{u}_4, \\
\ontuples{\pi}(\vect{c}) &= \vect{u}_1 c_{\pi(p)} \vect{u}_2 c_{\pi(q)} \vect{u}_3 c_{\pi(r)} \vect{u}_4.
\end{align*}
Moreover,
\begin{multline*}
\{(a_{\pi(p)}, b_{\pi(p)}, c_{\pi(p)}), (a_{\pi(q)}, b_{\pi(q)}, c_{\pi(q)}), (a_{\pi(r)}, b_{\pi(r)}, c_{\pi(r)})\}
\\
= \{(\alpha, \alpha, \alpha), (\beta, \beta, \alpha), (\alpha, \alpha, \beta)\}
\end{multline*}
and $\alpha \beta \vect{u}_1 \vect{u}_2 \vect{u}_3 \vect{u}_4 \in A^\sharp$
and $\supp(\alpha \beta \vect{u}_1 \vect{u}_2 \vect{u}_3 \vect{u}_4) = L$.
Let $M := \ms(\vect{u}_1 \vect{u}_2 \vect{u}_3 \vect{u}_4)$.
Then $\ms(\vect{a}) = \ms(\ontuples{\pi}(\vect{a})) = \ms(\vect{c}) = \ms(\ontuples{\pi}(\vect{c})) = \multiset{\alpha, \alpha, \beta} \dotcup M$ and $\ms(\vect{b}) = \ms(\ontuples{\pi}(\vect{b})) = \multiset{\alpha, \beta, \beta} \dotcup M$.
Concerning $\singles(\ontuples{\pi}(\vect{a}))$, $\singles(\ontuples{\pi}(\vect{b}))$, and $\singles(\ontuples{\pi}(\vect{c}))$, we have different possibilities, depending on the relative order of $\pi(p)$, $\pi(q)$, and $\pi(r)$, as shown in the following table.

\begin{center}
\begin{tabular}{cccc}
\toprule
case & $\singles(\ontuples{\pi}(\vect{a}))$ & $\singles(\ontuples{\pi}(\vect{b}))$ & $\singles(\ontuples{\pi}(\vect{c}))$ \\
\midrule
$\pi(p) < \pi(q) < \pi(r)$ & $\vect{u}_1 \vect{u}_2 \beta \vect{u}_3 \vect{u}_4$ & $\vect{u}_1 \alpha \vect{u}_2 \vect{u}_3 \vect{u}_4$ & $\vect{u}_1 \vect{u}_2 \vect{u}_3 \beta \vect{u}_4$ \\
$\pi(p) < \pi(r) < \pi(q)$ & $\vect{u}_1 \vect{u}_2 \vect{u}_3 \beta \vect{u}_4$ & $\vect{u}_1 \alpha \vect{u}_2 \vect{u}_3 \vect{u}_4$ & $\vect{u}_1 \vect{u}_2 \beta \vect{u}_3 \vect{u}_4$ \\
$\pi(q) < \pi(p) < \pi(r)$ & $\vect{u}_1 \beta \vect{u}_2 \vect{u}_3 \vect{u}_4$ & $\vect{u}_1 \vect{u}_2 \alpha \vect{u}_3 \vect{u}_4$ & $\vect{u}_1 \vect{u}_2 \vect{u}_3 \beta \vect{u}_4$ \\
$\pi(q) < \pi(r) < \pi(p)$ & $\vect{u}_1 \beta \vect{u}_2 \vect{u}_3 \vect{u}_4$ & $\vect{u}_1 \vect{u}_2 \vect{u}_3 \alpha \vect{u}_4$ & $\vect{u}_1 \vect{u}_2 \beta \vect{u}_3 \vect{u}_4$ \\
$\pi(r) < \pi(p) < \pi(q)$ & $\vect{u}_1 \vect{u}_2 \vect{u}_3 \beta \vect{u}_4$ & $\vect{u}_1 \vect{u}_2 \alpha \vect{u}_3 \vect{u}_4$ & $\vect{u}_1 \beta \vect{u}_2 \vect{u}_3 \vect{u}_4$ \\
$\pi(r) < \pi(q) < \pi(p)$ & $\vect{u}_1 \vect{u}_2 \beta \vect{u}_3 \vect{u}_4$ & $\vect{u}_1 \vect{u}_2 \vect{u}_3 \alpha \vect{u}_4$ & $\vect{u}_1 \beta \vect{u}_2 \vect{u}_3 \vect{u}_4$ \\
\bottomrule
\end{tabular}
\end{center}

Let us now consider some specific values of $p$, $q$, $r$.
If we choose $p = 1$, $q = 2$, $r = 3$, then it holds that $\vect{u}_1 = \vect{u}_2 = \vect{u}_3 = \varepsilon$, so, regardless of the relative order of $\pi(1)$, $\pi(2)$, and $\pi(3)$, the equality $f'(\cs(\ontuples{\pi}(\vect{a}))) = f'(\cs(\ontuples{\pi}(\vect{b})))$ implies
\begin{equation}
f'(\multiset{\alpha, \alpha, \beta} \dotcup M, \beta \vect{u}_4)
= f'(\multiset{\alpha, \beta, \beta} \dotcup M, \alpha \vect{u}_4).
\label{eq:A1}
\end{equation}

Now choose $p = 1$, $q = n - 1$, $r = n$.
Then $\vect{u}_1 = \vect{u}_3 = \vect{u}_4 = \varepsilon$,
and, depending on the relative order of $\pi(1)$, $\pi(2)$, and $\pi(3)$, the equalities $f'(\cs(\ontuples{\pi}(\vect{a}))) = f'(\cs(\ontuples{\pi}(\vect{b}))) = f'(\cs(\ontuples{\pi}(\vect{c})))$ imply
either
\begin{equation}
f'(\multiset{\alpha, \alpha, \beta} \dotcup M, \vect{u}_2 \beta) = f'(\multiset{\alpha, \beta, \beta} \dotcup M, \alpha \vect{u}_2)
\label{eq:A2}
\end{equation}
or
\begin{equation}
f'(\multiset{\alpha, \alpha, \beta} \dotcup M, \beta \vect{u}_2) = f'(\multiset{\alpha, \beta, \beta} \dotcup M, \vect{u}_2 \alpha).
\label{eq:A3}
\end{equation}

Since the choice of the tuples $\vect{a}$, $\vect{b}$, $\vect{c}$ was arbitrary, it follows from identities \eqref{eq:A1}, \eqref{eq:A2}, \eqref{eq:A3} that the equalities
\begin{align}
f'(\multiset{\alpha, \alpha, \beta} \dotcup \ms(\vect{v}), \beta \vect{v}) &= f'(\multiset{\alpha, \beta, \beta} \dotcup \ms(\vect{v}), \alpha \vect{v}),
\label{eq:B1}
\\
f'(\multiset{\alpha, \alpha, \beta} \dotcup \ms(\vect{v}), \beta \vect{v}) &= f'(\multiset{\alpha, \beta, \beta} \dotcup \ms(\vect{v}), \vect{v} \alpha)
\label{eq:B2}
\end{align}
hold for any $\alpha, \beta \in A$ and any $\vect{v} \in A^{n-3}$ satisfying $\alpha \beta \vect{v} \in \Anneq[n-1]$.

The sets $\Anneq[n-1]$ and $\cs(T) := \{\cs(\vect{a}) \mid \vect{a} \in T\}$ are in one\hyp{}to\hyp{}one correspondence via the map $h \colon \Anneq[n-1] \to \cs(T)$ given by
\begin{align*}
(a_1, a_2, \dots, a_{n-1}) \mapsto & \cs(a_1, a_1, a_2, \dots, a_{n-1}) \\ &{} = (\multiset{a_1, a_1, a_2, \dots, a_{n-1}}, a_2 a_3 \dots a_{n-1}).
\end{align*}
Equalities \eqref{eq:B1} and \eqref{eq:B2} thus express the fact that $f'(h(\ontuples{\tau}(\vect{a}))) = f'(h(\vect{a}))$ and \linebreak $f'(h(\ontuples{\sigma}(\vect{a}))) = f'(h(\vect{a}))$ for all $\vect{a} \in T$, where $\sigma = (1 \; 2 \; \dots \; n-1)$, $\tau = (1 \; 2)$.
Since $\{\sigma, \tau\}$ generates the full symmetric group $\symm{n-1}$, this implies that $f' \circ h \colon \Anneq[n-1] \to B$ is totally symmetric.
It is easy to see that if $\vect{w}, \vect{w}' \in T$ satisfy $\supp(\vect{w}) = \supp(\vect{w}')$, then $h^{-1}(\cs(\vect{w})) = \ontuples{\pi}(h^{-1}(\cs(\vect{w})))$ for some $\pi \in \symm{n-1}$.
Consequently,
\begin{align*}
f(\vect{w})
&
= f'(\cs(\vect{w}))
= f'(h(h^{-1}(\cs(\vect{w}))))
= f'(h(\ontuples{\pi}(h^{-1}(\cs(\vect{w}'))))
\\ &
= f'(h(h^{-1}(\cs(\vect{w}'))))
= f'(\cs(\vect{w}'))
= f(\vect{w}'),
\end{align*}
where the fourth equality holds by the total symmetry of $f' \circ h$.
This means that $f|_T(\vect{w})$ only depends on $\supp(\vect{w})$, i.e., $f|_T$ is determined by $\supp$.
\end{proof}

\begin{proposition}
\label{prop:2st-and-cs}
Let $f \colon A^n \to B$.
Then $f$ determined by $\cs$ and $f|_{\Aneq}$ is $2$\hyp{}set\hyp{}transitive if and only if $f|_{\Aneq}$ is totally symmetric.
\end{proposition}

\begin{proof}
Assume that $f = f^* \circ {\cs}|_{A^n}$ and $f|_{\Aneq}$ is $2$\hyp{}set\hyp{}transitive.
Let $M = (A, \chi_M)$ be a multiset over $A$ with $\card{M} = n$ and assume that $M$ has an element $b$ of multiplicity at least $2$.
Let $T = \{a_1, \dots, a_\ell\}$ be the set of elements whose multiplicity in $M$ is equal to $1$, and
fix a string $\vect{u} \in A^{n - \ell - 2}$ such that $\ms(a_1 \dots a_\ell b b \vect{u}) = M$.

Let $i \in \nset{\ell - 1}$.
Since $f|_{\Aneq}$ is $2$\hyp{}set\hyp{}transitive, there exists a permutation $\sigma \in \Inv f|_{\Aneq}$ such that $\{\sigma(1), \sigma(2)\} = \{i, i + 2\}$.
Consider first the case that $\sigma(1) = i$ and $\sigma(2) = i + 2$.
Let $\vect{v} \in A^{n-2}$ be the unique string satisfying
\[
(a_1 \dots a_{i-1} a_i a_{i+1} b a_{i+2} \dots a_\ell b \vect{u}) \sigma = a_i b \vect{v}.
\]
Then we have
\begin{multline*}
f^*(M, a_1 \dots a_\ell)
= f(a_1 \dots a_{i-1} a_i a_{i+1} b a_{i+2} \dots a_\ell b \vect{u})
= f(a_i b \vect{v})
= f(b a_i \vect{v}) \\
= f(a_1 \dots a_{i-1} b a_{i+1} a_i a_{i+2} \dots a_\ell b \vect{u})
= f^*(M, a_1 \dots a_{i-1} a_{i+1} a_i a_{i+2} \dots a_\ell),
\end{multline*}
In the above, the first and the last equalities hold because $f = f^* \circ {\cs}|_{A^n}$, the second and the fourth equalities hold because $\sigma \in \Inv f|_{\Aneq}$, and the third equality holds because $f$ is determined by $\cs$ and clearly $\cs(a_i b \vect{v}) = \cs(b a_i \vect{v})$ since $b$ occurs at least twice in $a_i b \vect{v}$.
A similar argument shows that $f^*(M, a_1 \dots a_\ell) = f^*(M, a_1 \dots a_{i-1} a_{i+1} a_i a_{i+2} \dots a_\ell)$ also in the case when $\sigma(1) = i + 2$ and $\sigma(2) = i$.

Since the full symmetric group $\symm{\ell}$ is generated by the adjacent transpositions $(i \;\; i + 1)$, $1 \leq i \leq \ell - 1$, it follows that $f^*(M, a_1 \dots a_\ell) = f^*(M, a_{\pi(1)} \dots a_{\pi(\ell)})$ for any permutation $\pi \in \symm{\ell}$.
In other words, $f^*$ does not depend on its second argument when it is restricted to the set of pairs $(M,\vect{u}) \in \mathcal{Z}^{(n)}(A)$ where $M$ is a multiset containing an element of multiplicity at least $2$.
We conclude that $f|_{\Aneq}$ is totally symmetric.

For the converse implication, assume that $f|_{\Aneq}$ is totally symmetric.
Then obviously $f|_{\Aneq}$ is $2$\hyp{}set\hyp{}transitive.
Furthermore, it is easy to see that $f = f^* \circ {\cs}|_{A^n}$, where $f^* \colon \mathcal{Z}(A) \to B$ is defined as
$f^*(M,\vect{a}) = f(\vect{a}\vect{b})$ where $\vect{b} \in A^*$ is any string such that $\ms(\vect{a} \vect{b}) = M$.
The mapping $f^*$ is well defined, because the total symmetry of $f|_{\Aneq}$ implies that $f(\vect{a} \vect{b})$ does not depend on the particular choice of $\vect{b}$.
\end{proof}

Let us introduce some notation for classes of functions having a unique identification minor:
\begin{itemize}
\item $\clUNI := \{f \in \mathcal{F}_{AB} : \text{$f$ has a unique identification minor}\}$,
\item $\clOFO := \{f \in \mathcal{F}_{AB} : \text{$f \simeq g$ for some $g$ that is determined by $\ofo$}\}$,
\item $\clCS := \{f \in \mathcal{F}_{AB} : \text{$f \simeq g$ for some $g$ that is determined by $\cs$}\}$,
\item $\clTWOST := \{f \in \mathcal{F}_{AB} : \text{$f$ is $2$-set-transitive}\}$,
\item $\clSYMM := \{f \in \mathcal{F}_{AB} : \text{$f$ is totally symmetric}\}$,
\item $\clSUPP := \{f \in \mathcal{F}_{AB} : \text{$f$ is determined by $\supp$}\}$.
\end{itemize}

The relationships between these classes can be summarized as follows; see also Figure~\ref{fig:inclusions}.

\begin{corollary}
Assume that $\card{A} = k$ and $\card{B} \geq 2$.
\begin{enumerate}[\upshape (i)]
\item\label{CS=SYMM-2} If $k = 2$, then $\clCS^{(n)} = \clSYMM^{(n)}$ for every $n \geq 3$.

\item\label{CS-nsub-OFO-2ST} $\clCS^{(n)} \nsubseteq \clOFO^{(n)} \cup \clTWOST^{(n)}$ for every $n \geq k + 1$.

\item\label{2ST-OFO-nsub-CS} $\clOFO^{(n)} \nsubseteq \clCS^{(n)}$ and $\clTWOST^{(n)} \nsubseteq \clCS^{(n)}$ for every $n \geq k+1$.

\item\label{OFO-CS} $\clOFO^{(n)} \cap \clCS^{(n)} = \clSUPP^{(n)}$ for every $n \geq k + 1$.

\item\label{2ST-CS} $\clTWOST^{(n)} \cap \clCS^{(n)} = \clSYMM^{(n)}$ for every $n \geq k + 1$.

\item\label{2ST-OFO} $\clTWOST^{(n)} \cap \clOFO^{(n)} = \clSUPP^{(n)}$ for every $n \geq k + 2$
\end{enumerate}
\end{corollary}

\begin{proof}
Statement \eqref{CS=SYMM-2} follows from Lemma~\ref{lem:cs-new-2}, \eqref{CS-nsub-OFO-2ST} from Lemma~\ref{lem:cs-new}, \eqref{2ST-OFO-nsub-CS} from Propositions~\ref{prop:ofo-and-cs} and~\ref{prop:2st-and-cs}, and \eqref{2ST-OFO} from Proposition~\ref{prop:2st-ofo}.
If $n > \card{A}$, then $\Aneq = A^n$, and Propositions~\ref{prop:ofo-and-cs} and \ref{prop:2st-and-cs} reduce to statements \eqref{OFO-CS} and \eqref{2ST-CS}, respectively.
\end{proof}

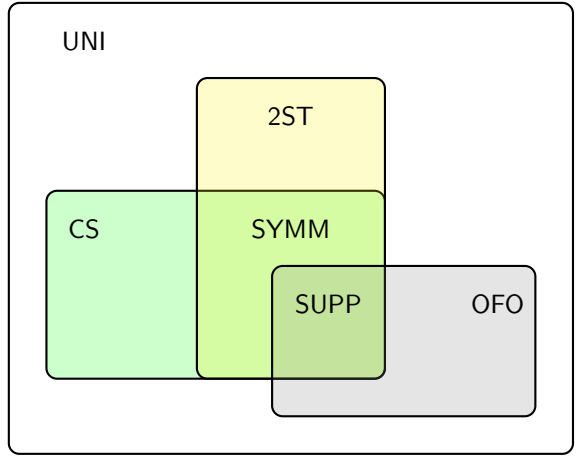
\begin{figure}
\begin{tikzpicture}[thick,draw opacity=1,fill opacity=0.2,text opacity=1,scale=1]
\filldraw[fill=green,rounded corners] (-4.5,0.5) rectangle (0,3);
\filldraw[fill=yellow,rounded corners] (-2.5,0.5) rectangle (0,4.5);
\filldraw[fill=gray,rounded corners] (-1.5,0) rectangle (2,2);
\draw[rounded corners] (-5,-0.5) rectangle (2.5,5.5);
\node[opacity=1] at (-4,5) {$\clUNI$};
\node[opacity=1] at (-1.25,4) {$\clTWOST$};
\node[opacity=1] at (-1.25,2.5) {$\clSYMM$};
\node[opacity=1] at (-4,2.5) {$\clCS$};
\node[opacity=1] at (-0.75,1.5) {$\clSUPP$};
\node[opacity=1] at (1.5,1.5) {$\clOFO$};
\end{tikzpicture}
\caption{Inclusions between classes of $n$\hyp{}ary functions with a unique identification minor, for $n \geq \card{A} + 2$ and $\card{A} \geq 3$.}
\label{fig:inclusions}
\end{figure}


\section{How do content and singletons and order of first occurrence arise from unique identification minors?}
\label{sec:how}

The classes $\clOFO$ and $\clCS$ may seem to have arisen just by accident as sporadic examples of functions with a unique identification minor.
We would like to understand better why these examples exist.
For this reason, we go back to the definition of a function with a unique identification minor and try to see how the classes $\clOFO$ and $\clCS$ arise therefrom by imposing some simple additional conditions.

Recall that a function $f \colon A^n \to B$ has a unique identification minor if and only if there exist a function $h \colon A^{n-1} \to B$ and a family $(\rho_I)_{I \in \couples}$ of permutations $\rho_I \in \symm{n-1}$ such that $f_I = h \circ \ontuples{\rho_I}$ for every $I \in \couples$.
We are going to see that $f$ belongs to one of the classes $\clOFO$ or $\clCS$ if and only if the family $(\rho_I)_{I \in \couples}$ of permutations satisfies certain conditions.

Let $\vect{a}, \vect{b} \in A^n$.
We write $\vect{a} \sim \vect{b}$ if there exist $\vect{u} \in A^{n-1}$ and $I, J \in \couples$ such that $\vect{a} = \vect{u} \delta_I$ and $\vect{b} = \vect{u} \delta_J$.
The relation $\sim$ is clearly reflexive and symmetric.
Denote by $\sim^*$ the transitive closure of $\sim$.

\begin{lemma}
\label{lem:ab-ofo}
For all $\vect{a}, \vect{b} \in A^n$, it holds that $\vect{a} \sim^* \vect{b}$ if and only if $\ofo(\vect{a}) = \ofo(\vect{b})$.
\end{lemma}

\begin{proof}
Let us first prove necessity.
If $\vect{a} \sim \vect{b}$, then $\vect{a} = \vect{u} \delta_I$ and $\vect{b} = \vect{u} \delta_J$ for some $\vect{u} \in A^{n-1}$ and $I, J \in \couples$.
Then $\ofo(\vect{a}) = \ofo(\vect{u} \delta_I) = \ofo(\vect{u}) = \ofo(\vect{u} \delta_J) = \ofo(\vect{b})$.
From this, it immediately follows that if $\vect{a} \sim^* \vect{b}$, then $\ofo(\vect{a}) = \ofo(\vect{b})$.

For sufficiency, observe first that for any $a, b \in A$, $\vect{x}, \vect{y}, \vect{z} \in A^*$ with $\card{\vect{x}} = p$, $\card{\vect{y}} = q$, it holds that
\begin{align*}
& (a \vect{x} b \vect{y} \vect{z}) \delta_{\{p + 2, p + q + 3\}} = a \vect{x} b \vect{y} b \vect{z} &
& (a \vect{x} \vect{y}) \delta_{\{1, p + 2\}} = a \vect{x} a \vect{y}, \\
& (a \vect{x} b \vect{y} \vect{z}) \delta_{\{1, 2\}} = a a \vect{x} b \vect{y} \vect{z} &
& (a \vect{x} \vect{y}) \delta_{\{1, 2\}} = a a \vect{x} \vect{y}.
\end{align*}
Therefore $a \vect{x} b \vect{y} b \vect{z} \sim a a \vect{x} b \vect{y} \vect{z}$ and $a \vect{x} a \vect{y} \sim a a \vect{x} \vect{y}$.
Applying these relations repeatedly, we see that for any $\vect{w} \in A^n$,
\[
\vect{w} \sim^* \underbrace{a \dots a}_{\ell} \ofo(\vect{w}),
\]
where $\ell = n - \card{\ofo(\vect{w})}$ and $a$ is the first letter of $\ofo(\vect{w})$.
Consequently, if $\vect{a}, \vect{b} \in A^n$ satisfy $\ofo(\vect{a}) = \ofo(\vect{b})$ and the first letter of $\ofo(\vect{a})$ is $a$, then
\[
\vect{a} \sim^* a \dots a \ofo(\vect{a}) = a \dots a \ofo(\vect{b}) \sim^* \vect{b}.
\qedhere
\]
\end{proof}

In the following proposition, the family $(\rho_I)_{I \in \couples}$ can be thought of as comprising only identity permutations.

\begin{proposition}
Let $f \colon A^n \to B$.
Then $f$ is determined by the order of first occurrence if and only if there exists a function $h \colon A^{n-1} \to B$ such that $f_I = h$ for all $I \in \couples$.
\end{proposition}

\begin{proof}
Necessity is proved in~\cite[Proposition~4]{Lehtonen-unique}: if $f = f^* \circ {\ofo}|_{A^n}$, then $f_I = f^* \circ {\ofo}|_{A^{n-1}}$ for all $I \in \couples$.

For sufficiency, assume that $f_I = h$ for all $I \in \couples$. We need to show that $f(\vect{a}) = f(\vect{b})$ whenever $\ofo(\vect{a}) = \ofo(\vect{b})$.
Let $\vect{a}, \vect{b} \in A^n$ and assume that $\ofo(\vect{a}) = \ofo(\vect{b})$.
Then $\vect{a} \sim^* \vect{b}$ by Lemma~\ref{lem:ab-ofo}, so there exist tuples $\vect{c}^0, \dots, \vect{c}^r$ such that $\vect{a} = \vect{c}^0 \sim \vect{c}^1 \sim \dots \sim \vect{c}^{r-1} \sim \vect{c}^r = \vect{b}$. Then for each $i \in \nset{r}$ there exist $\vect{u}^i \in A^{n-1}$ and $I_i, J_i \in \couples$ such that $\vect{c}^{i-1} = \vect{u}^i \delta_{I_i}$ and $\vect{c}^i = \vect{u}^i \delta_{J_i}$. Since
\[
f(\vect{c}^{i-1})
= f(\vect{u}^i \delta_{I_i})
= f_{I_i}(\vect{u}^i)
= h(\vect{u}^i)
= f_{J_i}(\vect{u}^i)
= f(\vect{u}^i \delta_{J_i})
= f(\vect{c}^i),
\]
it follows that $f(\vect{a}) = f(\vect{c}^0) = f(\vect{c}^1) = \dots = f(\vect{c}^r) = f(\vect{b})$.
\end{proof}

Let $\vect{a}, \vect{b} \in A^n$.
We write $\vect{a} \sim_2 \vect{b}$ if there exist $\vect{x}, \vect{y}, \vect{z}, \vect{x}', \vect{y}', \vect{z}', \in A^*$, $\alpha \in A$ such that $\vect{a} = \vect{x} \alpha \vect{y} \alpha \vect{z}$, $\vect{b} = \vect{x}' \alpha \vect{y}' \alpha \vect{z}'$, and $\vect{x} \vect{y} \vect{z} = \vect{x}' \vect{y}' \vect{z}'$.
The relation $\sim_2$ is clearly reflexive and symmetric.
Denote by $\sim_2^*$ the transitive closure of $\sim_2$.

\begin{lemma}
\label{lem:ab-cs}
For all $\vect{a}, \vect{b} \in A^n$, it holds that $\vect{a} \sim_2^* \vect{b}$ if and only if $\cs(\vect{a}) = \cs(\vect{b})$.
\end{lemma}

\begin{proof}
If $\vect{a} \sim_2 \vect{b}$, then clearly $\ms(\vect{a}) = \ms(\vect{b})$ and $\singles(\vect{a}) = \singles(\vect{b})$, i.e., $\cs(\vect{a}) = \cs(\vect{b})$.
It follows that $\cs(\vect{a}) = \cs(\vect{b})$ whenever $\vect{a} \sim_2^* \vect{b}$.

For the converse implication, let us introduce some notation. Assume that $A$ is equipped with a fixed linear order $\leq$.
For a tuple $\vect{u} \in A^n$, denote by $\vect{u}^\mathrm{pre}$ the longest prefix of $\vect{u}$ that contains only singletons of $\vect{u}$, and denote by $\vect{u}^\mathrm{suf}$ the longest suffix of $\vect{u}$ that is a non\hyp{}decreasing (with respect to $\leq$) sequence of letters that occur at least twice in $\vect{u}$.
It is possible that $\vect{u}^\mathrm{pre}$ and $\vect{u}^\mathrm{suf}$ may be empty.

Let $\vect{u} \in A^n$.
Unless $\vect{u} = \vect{u}^\mathrm{pre} \vect{u}^\mathrm{suf}$, we have $\vect{u} = \vect{u}^\mathrm{pre} \alpha \vect{w} \vect{u}^\mathrm{suf}$, where $\alpha$ is a letter occurring at least twice in $\vect{u}$.
If $\vect{w} = \vect{w}^1 \alpha \vect{w}^2$ for some strings $\vect{w}^1$ and $\vect{w}^2$, then let $\vect{u}'$, $\vect{u}''$ be strings such that $\vect{u}^\mathrm{suf} = \vect{u}' \vect{u}''$ and $\vect{u}' \alpha \alpha \vect{u}''$ is a non\hyp{}decreasing sequence, and let $\vect{v} = \vect{u}^\mathrm{pre} \vect{w}^1 \vect{w}^2 \vect{u}' \alpha \alpha \vect{u}''$.
Otherwise $\vect{u}^\mathrm{suf} = \vect{u}' \alpha \vect{u}''$ for some strings $\vect{u}'$ and $\vect{u}''$; in this case, let $\vect{v} = \vect{u}^\mathrm{pre} \vect{w} \vect{u}' \alpha \alpha \vect{u}''$.
In either case, we have $\vect{u} \sim_2 \vect{v}$ and, moreover, $\card{\vect{u}^\mathrm{pre}} \leq \card{\vect{v}^\mathrm{pre}}$ and $\card{\vect{u}^\mathrm{suf}} < \card{\vect{v}^\mathrm{suf}}$.

Thus, starting from an arbitrary tuple $\vect{u} \in A^n$, we can construct a finite sequence $\vect{c}^0, \vect{c}^1, \dots, \vect{c}^r$ such that $\vect{u} = \vect{c}^0 \sim_2 \vect{c}^1 \sim_2 \dots \sim_2 \vect{c}^{r-1} \sim_2 \vect{c}^r = \vect{u}^\dagger \vect{u}^\ddag$, where $\vect{u}^\dagger = \singles(\vect{u})$ and $\vect{u}^\ddag$ comprises those letters that occur in $\vect{u}$ at least twice, sorted in non\hyp{}decreasing order. Thus $\vect{u} \sim_2^* \vect{u}^\dagger \vect{u}^\ddag$.

Now, if $\vect{a}, \vect{b} \in A^n$ are tuples such that $\cs(\vect{a}) = \cs(\vect{b})$, then clearly $\vect{a}^\dagger = \vect{b}^\dagger$ and $\vect{a}^\ddag = \vect{b}^\ddag$.
Then $\vect{a} \sim_2^* \vect{a}^\dagger \vect{a}^\ddag$ and $\vect{b} \sim_2^* \vect{b}^\dagger \vect{b}^\ddag$.
Since $\vect{a}^\dagger \vect{a}^\ddag = \vect{b}^\dagger \vect{b}^\ddag$ and $\sim_2^*$ is a symmetric and transitive relation, we have $\vect{a} \sim_2^* \vect{b}$.
\end{proof}

\begin{proposition}
Let $f \colon A^n \to B$.
Then $f$ is determined by $\cs$ if and only if there exists a function $h \colon A^{n-1} \to B$ such that for every $I = \{i, j\} \in \couples$ \textup{(}$i < j$\textup{)} we have 
$f_I = h \circ \ontuples{\zeta_i}$, where $\zeta_i = (i \; i+1 \; \cdots \; n-1)$,
\end{proposition}

\begin{proof}
Necessity is established in the proof of Lemma~\ref{lem:cs-unique}.
Let us prove sufficiency.

Assume first that $\vect{a} \sim_2 \vect{b}$.
Then $\vect{a} = \vect{x} \alpha \vect{y} \alpha \vect{z}$ and $\vect{b} = \vect{x}' \alpha \vect{y}' \alpha \vect{z}'$ with $\vect{x} \vect{y} \vect{z} = \vect{x}' \vect{y}' \vect{z}'$.
Let $I = \{i, j\}$, $J = \{p, q\}$, where $i = \card{\vect{x}} + 1$, $j = \card{\vect{x} \vect{y}} + 2$, $p = \card{\vect{x}'}$, $q = \card{\vect{x}' \vect{y}'}$.
We have
\begin{multline*}
f(\vect{a})
= f(\vect{x} \alpha \vect{y} \alpha \vect{z})
= f_I(\vect{x} \alpha \vect{y} \vect{z})
= h((\vect{x} \alpha \vect{y} \vect{z}) \zeta_i)
= h(\vect{x} \vect{y} \vect{z} \alpha) \\
= h(\vect{x}' \vect{y}' \vect{z}' \alpha)
= h((\vect{x}' \alpha \vect{y}' \vect{z}') \zeta_p)
= f_J(\vect{x}' \alpha \vect{y}' \vect{z}')
= f(\vect{x}' \alpha \vect{y}' \alpha \vect{z}')
= f(\vect{b}).
\end{multline*}

Now, let $\vect{a}, \vect{b} \in A^n$, and assume that $\cs(\vect{a}) = \cs(\vect{b})$. Then $\vect{a} \sim_2^* \vect{b}$ by Lemma~\ref{lem:ab-cs}, so there exist tuples $\vect{c}^0, \dots, \vect{c}^r$ such that $\vect{a} = \vect{c}^0 \sim_2 \vect{c}^1 \sim_2 \cdots \sim_2 \vect{c}^{r-1} \sim_2 \vect{c}^r = \vect{b}$. As we have shown above, for each $i \in \nset{r}$ we have $f(\vect{c}^{i-1}) = f(\vect{c}^i)$, so it follows that $f(\vect{a}) = f(\vect{b})$.
\end{proof}


\section{Invariance groups of functions determined by content and singletons}
\label{sec:permutations}

We now focus on the symmetries of functions determined by $\cs$.
It is most natural to pose the following question.

\begin{question}
\label{q:inv-groups}
For $n \in \IN$, which subgroups of the symmetric group $\symm{n}$ are invariance groups of $n$\hyp{}ary functions determined by $\cs$?
\end{question}

In order to approach this problem, we will make use of the notion of permutation pattern, which we are now going to briefly recall (for further information, see, e.g., Bóna~\cite{Bona}, Kitaev~\cite{Kitaev}).
Any permutation $\pi \in \symm{n}$ corresponds to a string $\pi_1 \pi_2 \dots \pi_n$, where $\pi_i = \pi(i)$ for all $i \in \nset{n}$.
For any string $\vect{u}$ of distinct integers, the \emph{reduction} or \emph{reduced form} of $\vect{u}$, denoted by $\red(\vect{u})$, is the permutation obtained from $\vect{u}$ by replacing its $i$\hyp{}th smallest entry with $i$, for $1 \leq i \leq \card{\vect{u}}$.
A permutation $\tau \in \symm{\ell}$ is a \emph{pattern} (or an \emph{$\ell$\hyp{}pattern}) of a permutation $\pi \in \symm{n}$, or $\pi$ \emph{involves} $\tau$, denoted $\tau \leq \pi$, if there exists a substring $\vect{u} = \pi_{i_1} \pi_{i_2} \dots \pi_{i_\ell}$ of $\pi$ ($i_1 < i_2 < \dots < i_\ell$) such that $\red(\vect{u}) = \tau$.

Lehtonen and Pöschel~\cite{LehPos} formalized the notion of permutation pattern by making use of order\hyp{}isomorphisms and functional composition as follows.
For any $S \subseteq \nset{n}$ with $\card{S} = \ell$, let $h_S \colon \nset{\ell} \to S$ be the order\hyp{}isomorphism $(\nset{\ell}, {\leq}) \to (S, {\leq})$, where the two sets $\nset{\ell}$ and $S$ are equipped with the restriction of the natural order of natural numbers to the respective subsets of $\IN$.
For $\pi \in \symm{n}$, we define $\pi_S \colon \nset{\ell} \to \nset{\ell}$ as $\pi_S = h^{-1}_{\pi(S)} \circ \pi|_S \circ h_S$.
Being a composition of bijective maps, $\pi_S$ is clearly a permutation on $\nset{\ell}$.
Then the $\ell$-patterns of $\pi$ are precisely the permutations of the form $\pi_S$ for some $S \subseteq \nset{n}$ with $\card{S} = \ell$.

For $\pi \in \symm{n}$ and $\ell \leq n$, denote by $\Pat^{(\ell)} \pi$ the set of all $\ell$\hyp{}patterns of $\pi$, i.e., $\Pat^{(\ell)} \pi := \{\tau \in \symm{\ell} \mid \tau \leq \pi\} = \{\pi_S \mid S \subseteq \nset{n},\, \card{S} = \ell\}$.
For arbitrary subsets $S \subseteq \symm{\ell}$ and $T \subseteq \symm{n}$, let
\begin{align*}
\Pat^{(\ell)} T &:= \bigcup_{\pi \in T} \Pat^{(\ell)} \pi, \\
\Comp^{(n)} S &:= \{\pi \in \symm{n} \mid \Pat^{(\ell)} \pi \subseteq S\}.
\end{align*}
As shown by Lehtonen and Pöschel~\cite{LehPos}, the maps $\Comp^{(n)}$ and $\Pat^{(\ell)}$ constitute a monotone Galois connection between $\symm{\ell}$ and $\symm{n}$.
Furthermore, the operator $\Comp^{(n)}$ behaves well with respect to composition of permutations.

\begin{lemma}[{Lehtonen, Pöschel~\cite[Lemma~2.6]{LehPos}}]
Let $\pi, \tau \in \symm{n}$, let $\ell \in \nset{n}$, and let $S \subseteq \nset{n}$.
Then the following statements hold.
\begin{enumerate}[\upshape (i)]
\item $(\pi \circ \tau)_S = \pi_{\tau(S)} \circ \tau_S$.
\item $\Pat^{(\ell)} \pi \circ \tau \subseteq (\Pat^{(\ell)} \pi) \circ (\Pat^{(\ell)} \tau)$.
\end{enumerate}
\end{lemma}

This fact establishes the basis for a perhaps surprising connection between permutation patterns and permutation groups.

\begin{proposition}[{Lehtonen, Pöschel~\cite[Proposition~3.1]{LehPos}}]
\label{prop:CompG-group}
If $G$ is a subgroup of $\symm{\ell}$, then $\Comp^{(n)} G$ is a subgroup of $\symm{n}$.
\end{proposition}

For $\vect{a} \in A^n$, let us write
$\indexsingles(\vect{a}) := \{i \in \nset{n} \mid \forall j \in \nset{n} \colon a_i = a_j \implies i = j\}$.
In other words, the singletons of $\vect{a}$ are at the positions indexed by the elements of $\indexsingles(\vect{a})$.
Using this notation, we can write $\singles(\vect{a}) = \vect{a} h_{\indexsingles(\vect{a})}$.

\begin{lemma}
\label{lem:once-sigma}
For $\vect{a} \in A^n$ and $\sigma \in \symm{n}$, it holds that $\singles(\vect{a} \sigma) = \singles(\vect{a}) \sigma_S$, where $S = \sigma^{-1}(\indexsingles(\vect{a}))$.
\end{lemma}

\begin{proof}
Observe first that
\begin{align*}
\indexsingles(\vect{a} \sigma)
&= \{i \in \nset{n} \mid \forall j \in \nset{n} \colon a_{\sigma(i)} = a_{\sigma(j)} \implies \sigma(i) = \sigma(j)\} \\
&= \sigma^{-1}(\{i \in \nset{n} \mid \forall j \in \nset{n} \colon a_i = a_j \implies i = j\})
= \sigma^{-1}(\indexsingles(\vect{a})).
\end{align*}
Then we can write
\begin{align*}
\singles(\vect{a} \sigma)
&= \vect{a} \circ \sigma \circ h_{\indexsingles(\vect{a} \sigma)}
= \vect{a} \circ h_{\indexsingles(\vect{a})} \circ h_{\indexsingles(\vect{a})}^{-1} \circ \sigma \circ h_{\indexsingles(\vect{a} \sigma)} \\
&= \vect{a} \circ h_{\indexsingles(\vect{a})} \circ h_{\sigma(\sigma^{-1}(\indexsingles(\vect{a})))}^{-1} \circ \sigma \circ h_{\sigma^{-1}(\indexsingles(\vect{a}))}
= \singles(\vect{a}) \sigma_{\sigma^{-1}(\indexsingles(\vect{a}))},
\end{align*}
which gives the desired result.
\end{proof}

\begin{example}
Consider the tuples $\vect{a} = 12234555$ and $\vect{b} = 12324526$ and the permutation $\sigma = 54238617$.
Then $\vect{a} \sigma = 43225515$, $\vect{b} \sigma = 42236512$, $\indexsingles(\vect{a}) = \{1, 4, 5\}$, $\indexsingles(\vect{b}) = \{1, 3, 5, 6, 8\}$.
Let $S := \sigma^{-1}(\indexsingles(\vect{a})) = \{1, 2, 7\}$, $T := \sigma^{-1}(\indexsingles(\vect{b})) = \{1, 4, 5, 6, 7\}$.
Then $\sigma_S = \red(541) = 321$ and $\sigma_T = \red(53861) = 32541$.
We have
\begin{align*}
& \singles(\vect{a}) = 134,   & & \singles(\vect{a} \sigma) = 431,   & & 431 = 134 \sigma_S, \\
& \singles(\vect{b}) = 13456, & & \singles(\vect{b} \sigma) = 43651, & & 43651 = 13456 \sigma_T.
\end{align*}
\end{example}

Let $f = f^* \circ {\cs}|_{A^n}$, let $\sigma \in \symm{n}$, and let $f^\sigma := f \circ \ontuples{\sigma} = f^* \circ {\cs}|_{A^n} \circ \ontuples{\sigma}$, that is, $f^\sigma(a_1, \dots, a_n) = f^*(\cs(a_{\sigma(1)}, \dots, a_{\sigma(n)}))$.
Obviously $\ms(\vect{a} \sigma) = \ms(\vect{a})$, and by Lemma~\ref{lem:once-sigma} we have
$\singles(\vect{a} \sigma) = \singles(\vect{a}) \sigma_S$, where $S = \indexsingles(\vect{a})$.
Therefore, if $\cs(\vect{a}) = (M, \vect{u})$, then $\cs(\vect{a} \sigma) = (M, \vect{u} \sigma_S)$.

For $\ell \in Z(k,n)$ with $k = \card{A}$, we denote by $f^*_\ell$ the restriction of $f^*$ to the set $\mathcal{Z}^{(n)}_\ell(A)$.
We say that $\sigma \in \symm{\ell}$ is an \emph{invariant} of $f^*_\ell$ if $f^*_\ell(M, \vect{a}) = f^*_\ell(M, \vect{a} \sigma)$ for all $(M, \vect{a}) \in \mathcal{Z}^{(n)}_\ell(A)$.
Denote by $\Inv f^*_\ell$ the set of all invariants of $f^*_\ell$.
It is clear that $\Inv f^*_\ell$ is a permutation group, a subgroup of $\symm{\ell}$.

\begin{lemma}
\label{lem:symmetries}
Assume that $\card{A} = k$.
Let $f = f^* \circ {\cs}|_{A^n}$, $g = g^* \circ {\cs}|_{A^n}$ for some $f^*, g^* \colon \mathcal{Z}^{(n)}(A) \to B$, and let $\sigma \in \symm{n}$.
Then the following statements hold.
\begin{enumerate}[\rm (i)]
\item\label{lem:symmetries:1}
$g = f \circ \ontuples{\sigma}$ if and only if for every $\ell \in Z(k,n)$ and for every $\tau \in \Pat^{(\ell)} \sigma$, $g^*_\ell(M, \vect{a}) = f^*_\ell(M, \vect{a} \tau)$ for all $(M, \vect{a}) \in \mathcal{Z}^{(n)}_\ell(A)$.

\item\label{lem:symmetries:2}
$\sigma \in \Inv f$ if and only if $\Pat^{(\ell)} \sigma \subseteq \Inv f^*_\ell$ for every $\ell \in Z(k,n)$.

\item\label{lem:symmetries-Invf}
$\Inv f = \bigcap_{\ell \in Z(k,n)} \Comp^{(n)} \Inv f^*_\ell$.
\end{enumerate}
\end{lemma}

\begin{proof}
\begin{inparaenum}[\rm (i)]
\item
Assume first that $g = f \circ \ontuples{\sigma}$.
Let $\ell \in Z(k,n)$ and $\tau \in \Pat^{(\ell)} \sigma$. Then there exists $S \subseteq \nset{n}$ with $\card{S} = \ell$ such that $\tau = \sigma_S$.
Assume that $(M, \vect{a}) \in \mathcal{Z}^{(n)}_\ell(A)$, and let $\vect{b} \in A^n$ be a tuple such that $\ms(\vect{b}) = M$, $\singles(\vect{b}) = \vect{a}$, and $\indexsingles(\vect{b}) = \sigma(S)$.
Then $\sigma^{-1}(\indexsingles(\vect{b})) = S$, and $\singles(\vect{b} \sigma) = \vect{a} \sigma_S$ holds by Lemma~\ref{lem:once-sigma}.
Consequently,
\begin{multline*}
g^*_\ell(M, \vect{a})
= g^*(\ms(\vect{b}), \singles(\vect{b}))
= g(\vect{b}) \\
= f(\vect{b} \sigma)
= f^*(\ms(\vect{b} \sigma), \singles(\vect{b} \sigma))
= f^*_\ell(M, \vect{a} \sigma_S)
= f^*_\ell(M, \vect{a} \tau).
\end{multline*}

Assume then that for every $\ell \in Z(k,n)$ and for every $\tau \in \Pat^{(\ell)} \sigma$, the equality $g^*_\ell(M, \vect{a}) = f^*_\ell(M, \vect{a} \tau)$ holds for all $(M, \vect{a}) \in \mathcal{Z}^{(n)}_\ell(A)$.
Let $\vect{b} \in A^n$, and let $M := \ms(\vect{b})$, $\vect{a} := \singles(\vect{b})$, $\ell := \card{\vect{a}}$, $S := \sigma^{-1}(\indexsingles(\vect{b}))$.
Since $\sigma_S \in \Pat^{(\ell)} \sigma$ and
since $\vect{a} \sigma_S = \singles(\vect{b} \sigma)$ holds by Lemma~\ref{lem:once-sigma}, we have
\begin{multline*}
g(\vect{b})
= g^*(\ms(\vect{b}), \singles(\vect{b}))
= g^*_\ell(M, \vect{a}) \\
= f^*_\ell(M, \vect{a} \sigma_S)
= f^*(\ms(\vect{b} \sigma), \singles(\vect{b} \sigma))
= f(\vect{b} \sigma).
\end{multline*}
We conclude that $g = f \circ \ontuples{\sigma}$.

\item
Immediate consequence of part \eqref{lem:symmetries:1} with $f = g$.

\item
By part~\eqref{lem:symmetries:2}, the condition $\sigma \in \Inv f$ is equivalent to the condition that for all $\ell \in \nset{k}$, $\Pat^{(\ell)} \sigma \subseteq \Inv f^*_\ell$.
By the definition of $\Comp^{(n)}$, this in turn is equivalent to the condition that for all $\ell \in Z(k,n)$, $\sigma \in \Comp^{(n)} \Inv f^*_\ell$.
This is equivalent to $\sigma \in \bigcap_{\ell \in \nset{k}} \Comp^{(n)} \Inv f^*_\ell$.
\end{inparaenum}
\end{proof}

Lemma~\ref{lem:symmetries}\eqref{lem:symmetries-Invf} provides an answer to Question~\ref{q:inv-groups}.

\begin{theorem}
Assume that $A$ and $B$ are sets with $\card{A} = k \geq 2$ and $\card{B} \geq 2$.
Then a subgroup $G$ of $\symm{n}$ is the invariance group of a function $f \colon A^n \to B$ determined by $\cs$ if and only if there exists a family $(G_\ell)_{\ell \in Z(k,n)}$ of permutation groups $G_\ell \leq \symm{\ell}$ such that $G = \bigcap_{\ell \in Z(k,n)} \Comp^{(n)} G_\ell$.
\end{theorem}

\begin{proof}
It is clear that, for any $\ell \in Z(k,n)$, every subgroup of $\symm{\ell}$ is a possible invariance group of a function $f^*_\ell \colon \mathcal{Z}^{(n)}_\ell(A) \to B$.
The claim then follows immediately from Lemma~\ref{lem:symmetries}\eqref{lem:symmetries-Invf}.
\end{proof}

In order to provide a more explicit answer to Question~\ref{q:inv-groups}, we would need to know which permutation groups are of the form $\Comp^{(n)} G$ for some permutation group $G \leq \symm{\ell}$, $\ell \in Z(k,n)$.
Such groups were investigated by the current author in \cite{Lehtonen-Pat-Comp,LehPos}, and we provide a coarse summary of the relevant results from \cite{Lehtonen-Pat-Comp} here.

We will use the following notation for certain permutations in $\symm{n}$:
\begin{compactitem}
\item the \emph{identity permutation} $\asc{n} = 1 2 \ldots n$,
\item the \emph{descending permutation} $\desc{n} = n (n-1) \ldots 1$,
\item the \emph{natural cycle} $\natcycle{n} = (1 \; 2 \; \cdots \; n) = 2 3 \ldots n 1$.
\end{compactitem}
The following subgroups of the symmetric group $\symm{n}$ will appear in the statement of the results:
\begin{compactitem}
\item the trivial group $\{\asc{n}\}$,
\item the group $\gensg{\desc{n}} = \{\asc{n}, \desc{n}\}$ generated by the descending permutation $\desc{n}$,
\item the \emph{natural cyclic group} $\cycl{n} := \gensg{\natcycle{n}}$,
\item the \emph{natural dihedral group} $\dihed{n} := \gensg{\natcycle{n}, \desc{n}}$,
\item for $a, b \in \IN$ with $a + b \leq n$, the group $\symm{n}^{a,b}$ of permutations that map the set $\{1, \dots, a\}$ onto itself and map the set $\{n - b + 1, \dots, n\}$ onto itself and fix all remaining points.
\end{compactitem}

\begin{proposition}[\cite{Lehtonen-Pat-Comp}]
\label{prop:Compn-G}
Let $G$ be a subgroup of $\symm{\ell}$.
\begin{enumerate}[\upshape (i)]
\item If $G$ is neither an intransitive group nor an imprimitive group with $\natcycle{\ell} \notin G$,
then, for every $n \geq \ell + 2$, $\Comp^{(n)} G$ is one of the following groups:
$\symm{n}$,
$\dihed{n}$,
$\cycl{n}$,
$\gensg{\desc{n}}$,
$\{\asc{n}\}$.

\item Assume that $G$ is an intransitive group or an imprimitive group with $\natcycle{\ell} \notin G$.
Let $a$ and $b$ be the largest numbers $\alpha$ and $\beta$, respectively, such that $\symm{\ell}^{\alpha,\beta} \leq G$.
Then there exists a number $m$ such that $\Comp^{(n)} G$ equals either $\symm{n}^{a,b}$ or $\gensg{\symm{n}^{a,b}, \desc{n}}$ for all $n \geq \ell + m$.
The smallest such number $m$ satisfies $m \leq \ell - 1$ in the case when $G$ is intransitive, and $m \leq p$, where $p$ is the largest proper divisor of $n$, in the case when $G$ is imprimitive and $\natcycle{\ell} \notin G$.
\end{enumerate}
\end{proposition}

\begin{proof}
This theorem combines the main results of \cite{Lehtonen-Pat-Comp}:
\begin{inparaenum}[\upshape (i)]
\item Theorems 4.3, 4.5, 4.6, 4.16 in \cite{Lehtonen-Pat-Comp},
\item Theorems 4.12 and 4.15 in \cite{Lehtonen-Pat-Comp}.
\end{inparaenum}
\end{proof}

\begin{corollary}
\label{cor:invariance}
Assume that $\card{A} = k$.
Let $f = f^* \circ {\cs}|_{A^n}$ for some $f^* \colon \mathcal{Z}^{(n)}(A) \to B$.
If $n \geq 2k - 3$, then $\Inv f$ is one of the following groups:
\begin{equation}
\symm{n},
\dihed{n},
\cycl{n},
\gensg{\desc{n}},
\{\asc{n}\},
\symm{n}^{a,b},
\gensg{\symm{n}^{c,c}, \desc{n}}
\label{eq:possible-groups}
\end{equation}
where $a, b, c \in \IN_+$ with $a + b \leq k - 1$ and $2c \leq k - 1$.
\end{corollary}

\begin{proof}
By Lemma~\ref{lem:symmetries}\eqref{lem:symmetries-Invf},
$\Inv f = \bigcap_{\ell \in Z(k,n)} \Comp^{(n)} \Inv f^*_\ell$.
The invariance group $\Inv f^*_\ell$ is an arbitrary subgroup of $\symm{\ell}$, for each $\ell \in Z(k,n)$.
Proposition~\ref{prop:Compn-G} then implies that each one of the groups $\Comp^{(n)} \Inv f^*_\ell$ is among the groups listed in \eqref{eq:possible-groups}.
The set \eqref{eq:possible-groups} is closed under intersections, so we conclude that $\Inv f$ is a group in \eqref{eq:possible-groups}.
\end{proof}

Further information about permutation groups of the form $\Comp^{(n)} G$ can be found in \cite{Lehtonen-Pat-Comp,LehPos}.


\section{Distinct but similar functions determined by content and singletons}
\label{sec:similarity}

It might be possible that two distinct functions determined by $\cs$ are similar.
We now investigate the conditions under which this happens.

\begin{question}
\label{q:fg-cs}
Do there exist distinct functions $f^*, g^* \colon \mathcal{Z}^{(n)}(A) \to B$ such that $f^* \circ {\cs}|_{A^n} \simeq g^* \circ {\cs}|_{A^n}$?
Provide necessary and sufficient conditions for functions $f^*$ and $g^*$ to have this property.
\end{question}

The following simple group\hyp{}theoretical result will prove useful for approaching this problem.
Let $(G; \cdot)$ be a group. For an arbitrary nonempty subset $S$ of $G$, let $\differences S := S^{-1} S = \{x^{-1} y \mid x, y \in S\}$ be the set of differences of elements of $S$.

\begin{lemma}
\label{lem:generated-subgroups}
Let $G$ be a group, and let $S$ be a nonempty subset of $G$.
Then the following statements hold.
\begin{compactenum}[\upshape (i)]
\item\label{lem:generated-subgroups-1}
$\gensg{\differences S}$ is a subgroup of $\gensg{S}$.

\item\label{lem:generated-subgroups-2}
The following conditions are equivalent:
\begin{compactenum}[\upshape (a)]
\item $\gensg{\differences S} = \gensg{S}$,
\item $\gensg{\differences S} \cap S \neq \emptyset$,
\item $S \subseteq \gensg{\differences S}$.
\end{compactenum}
\end{compactenum}
\end{lemma}

\begin{proof}
\begin{inparaenum}[\rm (i)]
\item
By definition, every element of $\differences S$ belongs to the subgroup generated by $S$.
Hence $\differences S \subseteq \gensg{S}$, from which it follows that $\gensg{\differences S} \subseteq \gensg{S}$.

\item
(a) $\implies$ (b):
If $\gensg{\differences S} = \gensg{S}$, then obviously $\gensg{\differences S} \cap S = \gensg{S} \cap S = S \neq \emptyset$.

(b) $\implies$ (c):
Assume that $\gensg{\differences S} \cap S \neq \emptyset$.
Then there exists $a \in \gensg{\differences S} \cap S$.
Let $b \in S$.
Since $a \in \gensg{\differences S}$ and $a^{-1} b \in S^{-1} S = \differences S$, we have $b = a a^{-1} b \in \gensg{\differences S}$.
Hence $S \subseteq \gensg{\differences S}$.

(c) $\implies$ (a):
If $S \subseteq \gensg{\differences S}$, then clearly $\gensg{S} \subseteq \gensg{\differences S}$.
The converse inclusion $\gensg{\differences S} \subseteq \gensg{S}$ holds by part (i).
\end{inparaenum}
\end{proof}

\begin{lemma}
\label{lem:symmetries-2}
Assume that $\card{A} = k$.
Let $f = f^* \circ {\cs}|_{A^n}$, $g = g^* \circ {\cs}|_{A^n}$ for some $f^*, g^* \colon \mathcal{Z}^{(n)}(A) \to B$, and let $\sigma \in \symm{n}$.
Then the following statements hold.
\begin{enumerate}[\rm (i)]
\item\label{lem:symmetries:3}
If $\ell \in Z(k,n)$ and $\pi, \tau \in \symm{\ell}$ are permutations satisfying $g^*_\ell(M, \vect{a}) = f^*_\ell(M, \vect{a} \pi)$ and $g^*_\ell(M, \vect{a}) = f^*_\ell(M, \vect{a} \tau)$ for all $(M, \vect{a}) \in \mathcal{Z}^{(n)}_\ell(A)$, then $\pi^{-1} \tau \in \Inv f^*_\ell$.

\item\label{lem:symmetries:4}
If $g = f \circ \ontuples{\sigma}$, then $\gensg{\differences \Pat^{(\ell)} \sigma} \subseteq \Inv f^*_\ell$ for every $\ell \in Z(k,n)$.

\item\label{lem:symmetries:5}
If $g = f \circ \ontuples{\sigma}$ and $\Pat^{(\ell)} \sigma \subseteq \gensg{\differences \Pat^{(\ell)} \sigma}$ for every $\ell \in Z(k,n)$, then $f = g$.
\end{enumerate}
\end{lemma}

\begin{proof}
\begin{inparaenum}[\upshape (i)]
\item
For any $(M, \vect{a}) \in \mathcal{Z}^{(n)}_\ell(A)$, we have
\[
f^*_\ell(M, \vect{a})
= f^*_\ell(M, \vect{a} \pi^{-1} \pi)
= g^*_\ell(M, \vect{a} \pi^{-1})
= f^*_\ell(M, \vect{a} \pi^{-1} \tau),
\]
that is, $\pi^{-1} \tau \in \Inv f^*_\ell$.

\item
By Lemma~\ref{lem:symmetries}\eqref{lem:symmetries:1}, for every $\pi, \tau \in \Pat^{(\ell)} \sigma$ it holds that $f^*_\ell(M, \vect{a} \pi) = g^*_\ell(M, \vect{a}) = f^*_\ell(M, \vect{a} \tau)$
for all $(M, \vect{a}) \in \mathcal{Z}^{(n)}_\ell(A)$.
By part \eqref{lem:symmetries:3}, $\pi^{-1} \tau \in \Inv f^*_\ell$.
Therefore, every permutation in $\differences \Pat^{(\ell)} \sigma$ is an invariant of $f^*_\ell$ and hence $\gensg{\differences \Pat^{(\ell)} \sigma}  \subseteq \Inv f^*_\ell$.

\item
By our assumptions and part \eqref{lem:symmetries:4}, $\Pat^{(\ell)} \sigma \subseteq \gensg{\differences \Pat^{(\ell)} \sigma} \subseteq \Inv f^*_\ell$ for every $\ell \in Z(k,n)$.
Then $\sigma \in \Inv f$ by Lemma~\ref{lem:symmetries}\eqref{lem:symmetries:2}.
Consequently, $g = f \circ \ontuples{\sigma} = f$.
\end{inparaenum}
\end{proof}

\begin{proposition}
\label{prop:cs-distinct-equivalent}
Let $n, k \in \IN$, let $\sigma \in \symm{n}$, and let $A$, $B$ be sets with $\card{A} = k$, $\card{B} \geq 2$.
Then there exist functions $f = f^* \circ {\cs}|_{A^n}$ and $g = g^* \circ {\cs}|_{A^n}$ such that $f = g \circ \ontuples{\sigma}$ and $f \neq g$ if and only if $\gensg{\differences \Pat^{(\ell)} \sigma} \cap \Pat^{(\ell)} \sigma = \emptyset$ for some $\ell \in Z(k,n)$.
\end{proposition}

\begin{proof}
Assume first that $\gensg{\differences \Pat^{(\ell)} \sigma} \cap \Pat^{(\ell)} \sigma \neq \emptyset$ for every $\ell \in Z(k,n)$.
Then $\Pat^{(\ell)} \sigma \subseteq \gensg{\differences \Pat^{(\ell)} \sigma}$ for all $\ell \in Z(k,n)$ by Lemma~\ref{lem:generated-subgroups}\eqref{lem:generated-subgroups-2}.
Lemma~\ref{lem:symmetries-2}\eqref{lem:symmetries:5} then implies that if $f = f^* \circ {\cs}|_{A^n}$ and $g = g^* \circ {\cs}|_{A^n}$ satisfy $g = f \circ \ontuples{\sigma}$, then $f = g$.

Assume then that $\ell \in Z(k,n)$ is a number such that $\gensg{\differences \Pat^{(\ell)} \sigma} \cap \Pat^{(\ell)} \sigma = \emptyset$.
Define the functions $f^*, g^* \colon \mathcal{Z}^{(n)}(A) \to B$ as follows. Fix a permutation $\rho \in \Pat^{(\ell)} \sigma$, and let
\begin{align*}
f^*(M, \vect{a})
&=
\begin{cases}
1, & \text{if $\vect{a} = (1 \dots \ell) \pi$ for some $\pi \in \gensg{\differences \Pat^{(\ell)} \sigma}$,} \\
0, & \text{otherwise,}
\end{cases}
\\
g^*(M, \vect{a})
&=
\begin{cases}
1, & \text{if $\vect{a} = (1 \dots \ell) \pi \rho^{-1}$ for some $\pi \in \gensg{\differences \Pat^{(\ell)} \sigma}$,} \\
0, & \text{otherwise.}
\end{cases}
\end{align*}
Let $f := f^* \circ {\cs}|_{A^n}$ and $g := g^* \circ {\cs}|_{A^n}$.
In order to show that $g = f \circ \ontuples{\sigma}$, it suffices to verify that $f^*$, $g^*$, and $\sigma$ satisfy the conditions of Lemma~\ref{lem:symmetries}\eqref{lem:symmetries:1}.
Observe first that for every $m \in Z(k,n)$ with $m \neq \ell$, the equalities $g^*_m(M, \vect{a}) = 0 = f^*_m(M, \vect{a} \pi)$ obviously hold for every $(M, \vect{a}) \in \mathcal{Z}^{(n)}_m(A)$ and for every permutation $\pi \in \symm{m}$.

Let then $(M, \vect{a}) \in \mathcal{Z}^{(n)}_\ell(A)$.
If $\supp{\vect{a}} \neq \{1, \dots, \ell\}$, then clearly $g^*_\ell(M, \vect{a}) = 0 = f^*_\ell(M, \vect{a} \pi)$ for every permutation $\pi \in \symm{\ell}$.
We may thus assume that $\supp{\vect{a}} = \{1, \dots, \ell\}$, i.e., $\vect{a} = (1 \dots \ell) \alpha$ for some $\alpha \in \symm{\ell}$.

If $g^*_\ell(M, \vect{a}) = 1$, then $\vect{a} = (1 \dots \ell) \pi \rho^{-1}$ for some $\pi \in \gensg{\differences \Pat^{(\ell)} \sigma}$.
Then for every $\tau \in \Pat^{(\ell)} \sigma$, we have $\pi \rho^{-1} \tau \in \gensg{\differences \Pat^{(\ell)} \sigma}$, so $f^*_\ell(M, \vect{a} \tau) = 1$.

If $g^*_\ell(M, \vect{a}) = 0$, then
$\vect{a} = (1 \dots \ell) \alpha$ for some $\alpha \in \symm{\ell}$ such that $\alpha \neq \pi \rho^{-1}$ for all $\pi \in \gensg{\differences \Pat^{(\ell)} \sigma}$, that is, $\alpha \rho \notin \gensg{\differences \Pat^{(\ell)} \sigma}$.
Let $\tau \in \Pat^{(\ell)} \sigma$, and suppose, to the contrary, that $\alpha \tau \in \gensg{\differences \Pat^{(\ell)} \sigma}$. Since $\tau^{-1} \rho \in \differences \Pat^{(\ell)} \sigma$, we get $\alpha \rho = \alpha \tau \tau^{-1} \rho \in \gensg{\differences \Pat^{(\ell)} \sigma}$, a contradiction. Therefore, for every $\tau \in \Pat^{(\ell)} \sigma$, it holds that $\alpha \tau \notin \gensg{\differences \Pat^{(\ell)} \sigma}$, and hence $f^*_\ell(M, \vect{a} \tau) = 0$.
Thus the conditions of Lemma~\ref{lem:symmetries}\eqref{lem:symmetries:1} are satisfied, and we conclude that $g = f \circ \ontuples{\sigma}$.

It remains to show that $f \neq g$.
In order to see this, let $\vect{b} \in A^n$ be any tuple satisfying $\singles(\vect{b}) = (1, 2, \dots, \ell)$. Such a tuple $\vect{b}$ exists; for example, take $\vect{b} := (1, 2, \dots, \ell, \ell + 1, \ell + 1, \dots, \ell + 1)$ if $\ell < n$ (and hence $\ell \leq n - 2$ and $\ell < k$) or take $\vect{b} := (1, 2, \dots, \ell)$ if $\ell = n$.
Then $f(\vect{b}) = f^*(\ms(\vect{b}), 1 \dots \ell) = 1$, because $1 \dots \ell = (1 \dots \ell) \id$ and $\id \in \gensg{\differences \Pat^{(\ell)} \sigma}$.
On the other hand, since $\Pat^{(\ell)} \sigma \cap \gensg{\differences \Pat^{(\ell)} \sigma} = \emptyset$ and $\rho \in \Pat^{(\ell)} \sigma$, there is no permutation $\pi \in \gensg{\differences \Pat^{(\ell)} \sigma}$ such that $\pi \rho^{-1} = \id$; hence $g(\vect{b}) = g^*(\ms(\vect{b}), 1 \dots \ell) = 0$.
\end{proof}

In view of Proposition~\ref{prop:cs-distinct-equivalent}, the condition $\gensg{\differences \Pat^{(\ell)} \sigma} \cap \Pat^{(\ell)} \sigma \neq \emptyset$, or equivalently, $\Pat^{(\ell)} \sigma \subseteq \gensg{\differences \Pat^{(\ell)} \sigma}$, or, equivalently, $\gensg{\differences \Pat^{(\ell)} \sigma} = \gensg{\Pat^{(\ell)} \sigma}$, according to Lemma~\ref{lem:generated-subgroups}\eqref{lem:generated-subgroups-2}) seems relevant for approaching Question~\ref{q:fg-cs}.
Let $\sigma \in \symm{n}$ and $\ell \in \nset{n}$. We say that $\sigma$ is \emph{$\ell$\hyp{}equalizing} if $\gensg{\differences \Pat^{(\ell)} \sigma} \cap \Pat^{(\ell)} \sigma \neq \emptyset$; otherwise $\sigma$ is \emph{$\ell$\hyp{}differentiating.}
We say that $\sigma$ is \emph{equalizing} if it is $\ell$\hyp{}equalizing for every $\ell \in \nset{n}$, and we say that $\sigma$ is \emph{differentiating} if it is $\ell$\hyp{}differentiating for some $\ell \in \nset{n}$.

Let $\pi \in \symm{n}$ and $\tau \in \symm{m}$.
The \emph{direct sum} $\pi \oplus \tau$ and the \emph{skew sum} $\pi \ominus \tau$ of $\pi$ and $\tau$ are the permutations in $\symm{n+m}$ consisting of shifted copies of $\pi$ and $\tau$:
\begin{align*}
(\pi \oplus \tau)(i)
&=
\begin{cases}
\pi(i), & \text{if $1 \leq i \leq n$,} \\
n + \tau(i - n), & \text{if $n + 1 \leq i \leq n + m$,}
\end{cases}
\\
(\pi \ominus \tau)(i)
&=
\begin{cases}
m + \pi(i), & \text{if $1 \leq i \leq n$,} \\
\tau(i - n), & \text{if $n + 1 \leq i \leq n + m$.}
\end{cases}
\end{align*}

For $m \in \IN_+$, we denote the ascending $m$\hyp{}permutation $1 2 \dots m$ and the descending $m$\hyp{}permutation $m (m-1) \dots 1$ by $\asc{m}$ and $\desc{m}$, respectively.
For notational convenience, we agree that $\desc{0} \oplus \desc{m} = \desc{m} \oplus \desc{0} = \desc{m}$ and $\asc{0} \ominus \asc{m} = \asc{m} \ominus \asc{0} = \asc{m}$.

\begin{remark}
It is easy to find examples of $\ell$\hyp{}equalizing $n$\hyp{}permutations.
For example, every $n$\hyp{}permutation containing the pattern $\asc{\ell}$ is $\ell$\hyp{}equalizing, because $\asc{\ell} \in \differences \Pat^{(\ell)} \sigma$ for every permutation $\sigma \in \symm{n}$.
\end{remark}

\begin{proposition}
\label{prop:l-diff}
Let $\ell, n \in \IN_+$ with $\ell \leq n$.
Then the following statements hold.
\begin{enumerate}[\rm (i)]
\item\label{prop:l-diff:1}
For every $n \geq 1$, every $n$\hyp{}permutation is $1$\hyp{}equalizing.

\item\label{prop:l-diff:2}
For every $n \geq 2$, the only $2$\hyp{}differentiating $n$\hyp{}permutation is $\desc{n}$.

\item\label{prop:l-diff:3}
For $3 \leq \ell \leq n - 1$, the following $n$\hyp{}permutations are $\ell$\hyp{}differentiating:
\begin{itemize}
\item $\desc{m} \oplus \desc{n-m}$, for $0 \leq m \leq n-1$,
\item $\pi \ominus \desc{n-\ell} \ominus \tau$, for every $\pi \in \symm{p}$, $\tau \in \symm{q}$ with $p, q \geq 1$ and $p + q = \ell$.
\end{itemize}

\item\label{prop:l-diff:4}
The only $n$\hyp{}equalizing $n$\hyp{}permutation is $\asc{n}$.
\end{enumerate}
\end{proposition}

\begin{proof}
\begin{inparaenum}[\rm (i)]
\item
Trivial.

\item
Let $\sigma \in \symm{n}$.
If $\sigma = n (n-1) \dots 1$, then $\Pat^{(2)} \sigma = \{21\}$ and $\differences \Pat^{(2)} \sigma = \{12\}$; hence $\gensg{\Pat^{(2)} \sigma} = \symm{2}$ and $\gensg{\differences \Pat^{(2)} \sigma} = \{12\}$.
If $\sigma = 1 2 \dots n$, then $\Pat^{(2)} \sigma = \{12\}$ and $\differences \Pat^{(2)} \sigma = \{12\}$; hence $\gensg{\Pat^{(2)} \sigma} = \gensg{\differences \Pat^{(2)} \sigma} = \{12\}$.
Otherwise $\Pat^{(2)} \sigma = \{12, 21\}$ and $\differences \Pat^{(2)} \sigma = \{12, 21\}$; hence $\gensg{\Pat^{(2)} \sigma} = \gensg{\differences \Pat^{(2)} \sigma} = \symm{2}$.

\item
Assume first that $\sigma = \desc{m} \oplus \desc{n-m}$ for some $m \in \{0, \dots, n-1\}$.
It is easy to see that $\Pat^{(\ell)} \sigma \subseteq \{\desc{p} \oplus \desc{\ell-p} \mid 0 \leq p \leq \ell-1\}$. Hence $\differences \Pat^{(\ell)} \sigma \subseteq \{\asc{p} \ominus \asc{\ell-p} \mid 0 \leq p \leq \ell-1\}$.
The set $\{\asc{p} \ominus \asc{\ell-p} \mid 0 \leq p \leq \ell-1\} = \gensg{(1 \; 2 \; \cdots \; \ell)}$ constitutes a subgroup of $\symm{\ell}$ that is clearly disjoint from $\Pat^{(\ell)} \sigma$. Thus $\gensg{\differences \Pat^{(\ell)} \sigma} \cap \Pat^{(\ell)} \sigma = \emptyset$, so Lemma~\ref{lem:generated-subgroups} implies that $\sigma$ is $\ell$\hyp{}differentiating.

Assume then that $\sigma = \pi \ominus \desc{n-\ell} \ominus \tau$ for some $\pi \in \symm{p}$, $\tau \in \symm{q}$.
It is easy to see that
$\Pat^{(\ell)} \sigma \subseteq \symm{p} \ominus \symm{q} := \{\rho \ominus \gamma \mid \rho \in \symm{p}, \gamma \in \symm{q}\}$.
Consequently, $\differences \Pat^{(\ell)} \sigma \subseteq \symm{p} \oplus \symm{q} := \{\rho \oplus \gamma \mid \rho \in \symm{p}, \gamma \in \symm{q}\}$.
The set
$\symm{p} \oplus \symm{q}$
constitutes a subgroup of $\symm{\ell}$ that is disjoint from
$\symm{p} \ominus \symm{q}$.
In order to see this, observe that for any
$\phi \in \symm{p} \oplus \symm{q}$,
we have $\phi^{-1}(1) \in \{1, \dots, p\}$, while for any
$\phi \in \symm{p} \ominus \symm{q}$,
it holds that $\phi^{-1}(1) \in \{p+1, \dots, \ell\}$.
Therefore $\gensg{\differences \Pat^{(\ell)} \sigma} \cap \Pat^{(\ell)} \sigma = \emptyset$, so Lemma~\ref{lem:generated-subgroups} implies that $\sigma$ is $\ell$\hyp{}differentiating.

\item
Trivial.
\end{inparaenum}
\end{proof}

\begin{remark}
Note that the descending permutation $\desc{n}$ is $\ell$\hyp{}differentiating for every $\ell$ with $2 \leq \ell \leq n$.
\end{remark}

\begin{remark}
Proposition~\ref{prop:l-diff} gives only a sufficient condition for $\ell$\hyp{}differentiating permutations in the case when $3 \leq \ell \leq n-1$.
It remains an open problem to determine necessary and sufficient conditions in this case.
\end{remark}

The \emph{reverse} of a tuple $\vect{a} = (a_1, \dots, a_n)$ is $\rev{\vect{a}} := (a_n, \dots, a_1)$.
In other words, $\rev{\vect{a}} = \vect{a} \desc{n}$.
The \emph{reverse} of a function $f \colon A^n \to B$ is the function $\rev{f} \colon A^n \to B$ given by the rule $\rev{f}(\vect{a}) = f(\rev{\vect{a}})$ for all $\vect{a} \in A^n$.
In other words, $\rev{f} = f \circ \ontuples{\desc{n}}$.
Obviously $f = \rev{f}$ if and only if $\desc{n} \in \Inv f$.

It is noteworthy that the reverse of any function determined by $\cs$ is also determined by $\cs$. For, it is easy to verify that is $f = f^* \circ {\cs}|_{A^n}$ for some $f^* \colon \mathcal{Z}^{(n)}(A) \to B$, then $\rev{f} = f' \circ {\cs}|_{A^n}$, where $f' \colon \mathcal{Z}^{(n)}(A) \to B$ is given by the rule $f'(M, \vect{a}) = f^*(M, \rev{\vect{a}})$ for all $(M, \vect{a}) \in \mathcal{Z}^{(n)}(A)$. 
This can also be seen from Lemma~\ref{lem:symmetries}\eqref{lem:symmetries:1} by noting that $\rev{\vect{a}} = \vect{a} \desc{\ell}$ and $\Pat^{(\ell)} \desc{n} = \{\desc{\ell}\}$ for every $\ell \in \nset{n}$.
Thus we obtain the following partial answer to Question~\ref{q:fg-cs}:
If $f = f^* \circ {\cs}|_{A^n}$ and $\desc{n} \notin \Inv f$, then $f \simeq \rev{f}$, $f \neq \rev{f}$, and $\rev{f} = f' \circ {\cs}|_{A^n}$.
This fact suggests that there are a lot of functions determined by $\cs$ that are similar to another distinct function determined by $\cs$.


\section{Concluding remarks and open problems}
\label{sec:remarks}

The work reported in this paper suggests several directions for further research. We would like to indicate a few open problems.

Corollary~\ref{cor:invariance} specifies the possible invariance groups of functions determined by $\cs$ when the arity $n$ is sufficiently large in relation to the cardinality $k$ of the domain, namely $n \geq 2k - 3$.
Question~\ref{q:inv-groups} remains open when $n < 2k - 3$.
A conclusive answer would require a more refined analysis of the groups of the form $\Comp^{(n)} G$, where $G$ is an intransitive or imprimitive group.
The paper~\cite{Lehtonen-Pat-Comp} only provides upper and lower bounds in these cases.

\begin{problem}
Let $A$ and $B$ be sets with $\card{A} = k \geq 2$ and $\card{B} \geq 2$.
What are the possible invariance groups of functions $f \colon A^n \to B$ determined by $\cs$ when $n < 2k - 3$?
\end{problem}

Lemma~\ref{lem:generated-subgroups} suggests the following group\hyp{}theoretical problem.

\begin{problem}
Let $G$ be a group.
Which subsets $S$ of $G$ satisfy $\gensg{\differences S} = \gensg{S}$?
\end{problem}

The author suspects this may be a very difficult problem in full generality. In the context of the current paper, we are interested in a special case of this problem, in which $G$ is the symmetric group $\symm{\ell}$ and the subsets $S$ are of the form $\Pat^{(\ell)} \sigma$ for some permutation $\sigma \in \symm{n}$, $n > \ell$, and we speak of $\ell$\hyp{}equalizing and $\ell$\hyp{}differentiating permutations $\sigma$, depending on whether $\gensg{\differences \Pat^{(\ell)} \sigma} = \gensg{\Pat^{(\ell)} \sigma}$ holds or not.

Proposition~\ref{prop:l-diff} gives only a sufficient condition for the $\ell$\hyp{}differentiating permutations in the case when $\ell \geq 3$.
The author has made some computer experiments that suggest that the condition might also be necessary when $3 \leq \ell \leq n - 2$.
In the case when $3 \leq \ell = n - 1$, there seem to exist many other $\ell$\hyp{}differentiating permutations than the ones listed in Proposition~\ref{prop:l-diff}, but an explicit characterization eludes us.

\begin{problem}
For $1 \leq \ell \leq n$, characterize the permutations $\sigma \in \symm{n}$ that satisfy the condition $\gensg{\differences \Pat^{(\ell)} \sigma} \cap \Pat^{(\ell)} \sigma = \emptyset$.
\end{problem}

In~\cite{Lehtonen-totsymm}, the current author posed the following question:
Is every function (of sufficiently large arity) uniquely determined, up to permutation of arguments, by its identification minors?
In precise terms, the \emph{deck} of a function $f \colon A^n \to B$, denoted $\deck f$, is the multiset $\multiset{f_I / {\simeq} : I \in \couples}$ of its identification minors, considered up to similarity (permutation of arguments). We say that a function $g \colon A^n \to B$ is a \emph{reconstruction} of $f$ if $\deck f = \deck g$.
We say that $f$ is \emph{reconstructible} if for every reconstruction $g$ of $f$ it holds that $f \simeq g$.
A class $\mathcal{C} \subseteq \mathcal{F}_{AB}$ is \emph{reconstructible} is all its members are reconstructible, and $\mathcal{C}$ is \emph{weakly reconstructible} if for all $f, g \in \mathcal{C}$, the condition $\deck f = \deck g$ implies $f \simeq g$.
Several results, both positive and negative, on this reconstruction problem were established in~\cite{CouLehSch-Post,CouLehSch-Sperner,Lehtonen-ms}.

In continuation of the line of research, it is natural to ask whether and to what extent the functions determined by content and singletons are reconstructible from identification minors.

\begin{problem}
Is every function (of sufficiently large arity) determined by $\cs$ reconstructible?
Is the class of functions determined by $\cs$ weakly reconstructible?
\end{problem}

Of course, Problem~\ref{problem}, the main question that this paper addresses, remains open, and we would like to repeat it here.

\begin{problem}
Characterize the functions with a unique identification minor.
\end{problem}


\section*{Acknowledgments}

The author would like to thank Alan J. Cain, Maria Jo\~ao Gouveia, and Reinhard P\"oschel for many stimulating discussions.

\end{document}